\documentclass{cmslatex}
\usepackage[paperwidth=7in, paperheight=10in, margin=.875in]{geometry}
 \usepackage[backref,colorlinks,linkcolor=red,anchorcolor=green,citecolor=blue]{hyperref}
\usepackage{amsfonts,amssymb}
\usepackage{amsmath}
\usepackage{graphicx}
\usepackage{cite}
\usepackage{enumerate}
\usepackage{xcolor}
\usepackage[section]{placeins}
\usepackage{lipsum}
\usepackage{epstopdf}
\usepackage{algorithmic}
 \usepackage{dsfont}
  \usepackage{url}
  \usepackage{mathrsfs}
\usepackage{subcaption}  
\usepackage{geometry}

\renewcommand{\theequation}{\arabic{section}.\arabic{equation}}

\newcommand{\E}{\mathcal{K}}
\newcommand{\mbp}{\mathbf{p}}
\newcommand{\F}{\mathcal{F}}%
\newcommand{\sw}[1]{{\color{blue} #1}}

   \allowdisplaybreaks
\begin{document}
\title{Numerical schemes for a multi-species quantum BGK model\thanks{Received date, and accepted date (The correct dates will be entered by the editor).}}


\author{Gi-Chan Bae \thanks{Research institute of Mathematics, Seoul National University, Seoul 08826, Republic of Korea, (gcbae02@snu.ac.kr).} \and Marlies Pirner \thanks{University of M\"{u}nster, Mathematics M\"{u}nster, Einsteinstrasse 62 48149, M\"{u}nster, Germany, (marlies.pirner@uni-muenster.de).} \and Sandra Warnecke \thanks{University of W\"{u}rzburg, Emil-Fischer-Straße 40, 97074 W\"{u}rzburg, Germany, (sandra.warnecke@uni-wuerzburg.de).}}

         \pagestyle{myheadings} \markboth{Numerical schemes for a multi-species quantum BGK model}{Gi-Chan Bae, Marlies Pirner, Sandra Warnecke} \maketitle

          \begin{abstract}
               This work is devoted to the numerical implementation of the quantum Bhatnagar-Gross-Krook (BGK) model for gas mixtures consisting of classical and quantum particles (fermions, bosons). We consider the model proposed by Bae, Klingenberg, Pirner, and Yun in 2021 and implement an Implicit-Explicit (IMEX) scheme due to the stiffness of the collision operator.
               A major obstacle is updating the parameters of quantum local equilibrium, which requires computing by inverting the relation between density and energy at every grid point in space and time. We address this difficulty by using the Lagrange multiplier method to minimize a potential function subject to constraints defined by specific moment equalities.
               Moreover, we analyze the convergence of mean velocity and temperature between the species both analytically and numerically. When a quantum component is included, we observe that the converging quantity is physical temperature, not the kinetic temperature. This differs from the mixture of classical species.
          \end{abstract}


\section{Introduction}
In a kinetic description, the state of a dilute gas or plasma is given by a distribution function that prescribes the density of particles at each point in position-momentum phase space. The evolution of this distribution function is governed by the quantum Boltzmann equation which describes a balance of particle advection and binary collisions. The Boltzmann collision operator representing the interaction between particles preserves collision invariants (mass, momentum, and energy) and dissipates the mathematical entropy of the system. Unfortunately, the expense of evaluating this operator can be prohibitive. Indeed, its evaluation requires the calculation of a five-dimensional integral at every point in phase-space. Thus even with fast spectral methods \cite{mouhot2006fast,pareschi2000numerical,gamba2009spectral,gamba2017fast}, the collision operator is typically the dominant part of a kinetic calculation. The quantum modification of the celebrated Boltzmann equation was made in \cite{61,62,Nordheim} to incorporate the quantum effect that cannot be neglected for light molecules (such as Helium) at low temperature. Quantum Boltzmann equation is now fruitfully employed not just for low temperature gases, but in various circumstances such as scattering problem in solid \cite{Ashcroft,FJ} and electrons on energy band structure in semiconductor \cite{Jungel}. 

In the classical case, the Bhatnagar-Gross-Krook (BGK) operator is a widely used surrogate for the Boltzmann operator, modeling collisions through a simple relaxation mechanism. This simplification provides significant computational advantages while maintaining the conservation and entropy dissipation properties of the Boltzmann operator. Similarly, quantum BGK models are widely used in place of the quantum Boltzmann equation \cite{Ashcroft,FJ,Jungel,Madelung,RS,MRS}. The mathematical analysis of the mixture or quantum model is developed in various ways \cite{BaeYun,MR4064205,CCmix,velocity-dependent,LuF1,LuB1,velocity-dependent-theory}.


The BGK-type modeling for classical mixtures has been developed in various ways to accurately approximate the moments of the collision operator in \cite{AndriesAokiPerthame2002,Bobylev,Brull_2012,Garzo1989,Greene,gross_krook1956,hamel1965,Sofonea2001}. However, in the quantum case, unlike the classical BGK model for gas mixtures, the equilibrium coefficients of the local equilibrium for quantum multi-species gases are defined through highly nonlinear relations that are not explicitly solvable. A sufficient condition is proven in \cite{Quantum}, guaranteeing the existence of equilibrium coefficients so that the model shares the same conservation laws and H-theorem as the quantum Boltzmann equation.

In this paper, we present a numerical implementation of the quantum multi-species BGK model developed in \cite{Quantum}. When solving the quantum mixture BGK model numerically, a major obstacle is updating the parameters $(a_k,b_k,c_k)$ from the known density, momentum, and energy $(n_k,P_k,E_k)$ in the local quantum equilibrium for the $k$-th species with particle mass $m_k$:
\begin{align}\label{Q-Max}
\mathcal{K}_{k}(t,x,p)=\frac{1}{e^{m_ka_k\big|\frac{p}{m_k}-b_k\big|^2+c_k}+\tau_k},
\end{align}
for $\tau_k=+1,-1,0$ for fermions, bosons, and classical particles, respectively. To compute $c_k$, we need to take the inverse of a specific function (see \eqref{ci}) which becomes even more complicated in the multi-species case (refer to Theorem 2.1 in \cite{Quantum}). Despite this, after discretizing for time, updating the equilibrium from the previous step yields an equation for moments, which serves as the constraint for the Lagrange multiplier method in Section \ref{sec:generaltime}. Thus, we can compute the parameters without inverting the function. \newline
Our next objective is to identify the convergence of mean velocity and temperature between the species in the space-homogeneous case. We denote the kinetic temperature $T_k$ and physical temperature $\vartheta_k$ for the $k$-th species as follows: 
\begin{align*}
T_k = \frac{2}{3}\left(\frac{E_k}{ n_k}- \frac{1}{2} \frac{|P_k|^2}{m_k n_k^2}\right), \quad \vartheta_k = \frac{1}{2a_k},
\end{align*}
where $(n_k,P_k,E_k)$ are the macroscopic fields specified in \eqref{macros} and $a_k$ is specified in \eqref{Q-Max}. We prove that the velocities $\frac{P_k}{N_k}$ converge to each other exponentially fast. 
However, while the temperature $T_k$ also converges exponentially, it does so with some deviation (see Section \ref{sec:macro-equations}).
By numerical simulation in Section \ref{sec:results}, we illustrate that the converging quantity is the physical temperature, not the kinetic temperature, when a quantum component is included.

The remainder of this paper is organized as follows.
In Section \ref{sec:model}, 
we recall the multi-species quantum  BGK model from \cite{Quantum} for fermion-fermion, fermion-boson and boson-boson mixtures and its main important properties. 
In section \ref{sec:macro-equations}, we prove convergence rates of \textcolor{black}{mean velocities to common values and of kinetic temperatures with some deviation, respectively,} in the space-homogeneous case, which we will verify numerically later.
In Section \ref{sec:time}, we present the first- and second-order implicit-explicit time discretizations that are used in the paper.  We also introduce an optimization-based approach for the implicit evaluation of the BGK operator.  In Section \ref{sec:space}, we describe the space discretization.
In Section \ref{sec:discreteproperties}, we verify some structure preserving properties of the semi-discrete scheme.  In Section \ref{sec:velocity}, we introduce the momentum discretization and summarize the numerical implementation of the optimization algorithm introduced in Section \ref{sec:time}.
In Section \ref{sec:results}, we provide an array of numerical results that illustrate the properties of our scheme. 

\section{The multi-species quantum BGK model}
\label{sec:model}
We consider two distribution functions $f_1=f_1(x,p,t) \geq 0$ and $f_2=f_2(x,p,t) \geq 0$ for species with masses $m_1>0$ and $m_2>0$, respectively, with the phase space variables (position and momentum) $x \in \Omega \subset \mathbb{R}^3$ and $p \in \mathbb{R}^3$ and time $t\geq 0$. To be as general as possible, we generalize the quantum mixture BGK model in \cite{Quantum} describing a mixture of bosons and fermions to the mixture of gases including the interaction of quantum-classical particles:
\begin{align}
\begin{split}
\partial_t f_1+ \frac{p}{m_1} \cdot \nabla_x f_1  = \nu_{11}n_1(\mathcal{K}_{11}-f_1) + \nu_{12}n_2(\mathcal{K}_{12}-f_1), \\
    \partial_t f_2 + \frac{p}{m_2} \cdot \nabla_x f_2 = \nu_{22}n_2(\mathcal{K}_{22}-f_2) + \nu_{21}n_1(\mathcal{K}_{21}-f_2),
    \end{split}
    \label{BGK_quantum}
\end{align}
where $\mathcal{K}_{ij}$ is the local equilibrium describing the interactions of $i$-th and $j$-th component and $\nu_{ij} n_j$ for $ i,j=1,2$ are the collision frequencies. In the following, we always assume $\nu_{12} n_2= \nu_{21} n_1$ as it is also done for example in \cite{Pirner}. More explicitly, the local equilibria read 
for fermion $\tau=+1$, for boson $\tau=-1$, and for classical particle $\tau=0$:
\begin{align}
\begin{split}
\mathcal{K}_{11}=\frac{1}{e^{m_1a_1\big|\frac{p}{m_1}-b_1\big|^2+c_1}+\tau},\quad
\mathcal{K}_{12}=\frac{1}{e^{m_1a\big|\frac{p}{m_1}-b\big|^2+c_{12}}+\tau}, \cr
\mathcal{K}_{22}=\frac{1}{e^{m_2a_2\big|\frac{p}{m_2}-b_2\big|^2+c_2}+\tau'}, \quad
\mathcal{K}_{21}=\frac{1}{e^{m_2a\big|\frac{p}{m_2}-b\big|^2+c_{21}}+\tau'}.
\label{parameters}
\end{split}
\end{align}
We denote $\mathcal{K}$ as Fermi-Dirac distribution, Bose-Einstein distribution and Maxwellian for the case $\tau=+1,-1,0$, respectively. The equilibrium parameters $(a_i,b_i,c_i)$ and $(a,b,c_{12},c_{21})$ will be determined uniquely in a way such that the conservation laws and the entropy principle are satisfied.
Note that the model includes the following cases depending on the types of the particles
\begin{align*}
	(\tau,\tau')=\left\{
	\begin{array}{ll}
		(+1,+1)& \text{(fermion-fermion)}\cr
		(-1,-1)&\mbox{(boson-boson)}\cr
		(+1,-1)&\mbox{(fermion-boson)}\cr
		(+1,~0~)& \mbox{(fermion-classical)} \cr
		(~0,-1~)& \mbox{(classical-boson)} \cr
		(~0~,~0~)& \mbox{(classical-classical)}
		\end{array}\right.
\end{align*}
We define the number density of particles $n_i$, momentum $P_i$, energy $E_i$ of each species as
\begin{align}\label{macros}
n_i=\int_{\mathbb{R}^3}f_idp,\quad P_i=\int_{\mathbb{R}^3}f_ip ~dp, \quad E_i=\int_{\mathbb{R}^3}f_i\frac{|p|^2}{2m_i}dp,
\end{align}
and $m_in_i=N_i$.
The parameters $(a_i,b_i,c_i)$ in $\mathcal{K}_{ii}$ 
%
and $(a,b, c_{12}, c_{21})$ of $\mathcal{K}_{ij}$ for $(ij)=(12),(21)$ 
can be chosen to satisfy conservation of mass, momentum and energy in the intra-species interactions
\begin{align}\label{conserv 1}
	\int_{\mathbb{R}^3}\mathcal{K}_{ii}dp=n_i, \quad \int_{\mathbb{R}^3}\mathcal{K}_{ii}p ~dp=P_i, \quad \int_{\mathbb{R}^3}\mathcal{K}_{ii}\frac{|p|^2}{2m_i}dp=E_i, \quad (i=1,2).
\end{align}
and conservation of mass, total momentum and total energy in the inter-species interactions
\begin{align}\label{conserv 2}
\begin{split}
		&\int_{\mathbb{R}^3}\mathcal{K}_{12}dp=n_1, \qquad \int_{\mathbb{R}^3}\mathcal{K}_{21}dp=n_2, \cr
		&\nu_{12}n_2\left(\int_{\mathbb{R}^3}\mathcal{K}_{12}p~dp-P_1\right)+\nu_{21}n_1\left(\int_{\mathbb{R}^3}\mathcal{K}_{21}p ~dp-P_2\right)=0, \cr
		&\nu_{12}n_2\left(\int_{\mathbb{R}^3}\mathcal{K}_{12}\frac{|p|^2}{2m_1}dp-E_1\right)+\nu_{21}n_1\left(\int_{\mathbb{R}^3}\mathcal{K}_{21}\frac{|p|^2}{2m_2}dp-E_2\right)=0,
	\end{split}
\end{align}
see Theorem 2.1 in \cite{Quantum} for quantum-quantum mixture case with unit collision frequencies. The proof can be extended for any choice of mixture between classical particles, fermions and bosons for more general velocity independent collision frequencies in a straight-forward way. For later references, we restate that \textcolor{black}{$c_i$ is determined by inverting the following relation 
\begin{align}\label{ci}
\frac{\int \frac{1}{e^{|p|^2+c_i}+\tau}dp}{\left(\int \frac{|p|^2}{e^{|p|^2+c_i}+\tau}dp\right)^{3/5}}=\frac{n_i^{\frac{2}{5}}}{\left(\frac{2m_iE_i}{n_i}-\frac{|P_i|^2}{n_i^2}\right)^{\frac{3}{5}}},
\end{align}
and $a_i$, $b_i$ determined by
\begin{align}\label{ai}
a_i=\left({\int_{\mathbb{R}^3}\frac{1}{e^{|p|^2+c_i}+\tau}dp}\right)^\frac{2}{3}n_i^{-\frac{2}{3}}, \qquad b_i=\frac{P_i}{n_i},
\end{align}
for $(i=1,2).$ The quantities $a$ and $b$ are given by 
\begin{align}\label{abform}
	a=\left(\frac{m_1^{\frac{3}{2}}\int_{\mathbb{R}^3}\frac{|p|^2}{e^{|p|^2+c_{12}}+\tau}dp+m_2^{\frac{3}{2}}\int_{\mathbb{R}^3}\frac{|p|^2}{e^{|p|^2+c_{21}}+\tau'}dp}{2E_1+2E_2-\frac{|P_1+P_2|^2}{N_1+N_2}}\right)^{\frac{2}{5}},\quad
	b=\frac{P_1+P_2}{N_1+N_2},
\end{align}
and $c_{12}$ and $c_{21}$ are determined by a more complex implicit condition than \eqref{ci}, see Theorem 2.1 in \cite{Quantum} for the detailed expression.} 
Moreover, $c_{12}, c_{21}$ can be proven to satisfy the following relationship
\begin{align}
	\frac{\textcolor{black}{m_1^{\frac{3}{2}}}\eta_{\tau}(c_{12})}{\textcolor{black}{m_2^{\frac{3}{2}}}\eta_{\tau'}(c_{21})}=\frac{n_1}{n_2}, 
 \label{relationships}
\end{align}

In \cite{Quantum}, it is shown that 
 the distribution function in the fermion case remains bounded from above by $1$ for all times  $t\geq 0$ if it holds true for $t=0$. 
The model also satisfies an H-Theorem for the entropy \begin{align*}
H_{\tau,\tau'}(f_1,f_2)&=H_{\tau}(f_1) + H_{\tau'}(f_2)
\end{align*}
with
\begin{align*}
H_\tau(f)&=\begin{cases}\int_{\Omega} \int_{\mathbb{R}^3} f\ln fdp dx \quad \mbox{for} \quad \tau=0 \\ 
\int_{\Omega} \int_{\mathbb{R}^3} f\ln f+\tau^{-1}(1-\tau f)\ln (1-\tau f)dp dx \quad \mbox{for} \quad \tau=\pm1
\end{cases} 
\end{align*}
($\tau=+1$ for fermion, $\tau=-1$, for boson), see Theorem 2.1 in \cite{Quantum}.
For further purposes, we denote the integrand by $h_{\tau}$, \sw{i.e.} 
\begin{align} 
h_{\tau}(z)=z\ln z+\tau^{-1}(1-\tau z)\ln (1-\tau z),
\label{entropy_h}
\end{align}
for $z>0$ if $\tau=0,-1$ and $0<z<1$ if $\tau=+1$. 
Eventually, we highlight that the equilibrium distributions $\mathcal{K}_{ij}$ solve an entropy minimization problem with contraints which ensure
the conservation properties, see \cite{EMV} for the one-species case and the appendix \ref{sec2} for the mixture case. 

We present further properties of the model in the following.

\section{Macroscopic equations and convergence rates for the velocities and kinetic temperatures in the space-homogeneous case} \label{sec:macro-equations}
In this section, we derive the macroscopic Euler equations of the model \eqref{BGK_quantum} and prove convergence rates of the mean velocities and the kinetic temperatures.
\subsection{Macroscopic equations}
We denote the macroscopic fields of the inter-species local equilibrium $\mathcal{K}_{12}$ and $\mathcal{K}_{21}$ by 
\begin{align}\label{def12}
\begin{split}
&P_{12}=\int_{\mathbb{R}^3}\mathcal{K}_{12}~p~dp, \quad P_{21}=\int_{\mathbb{R}^3}\mathcal{K}_{21}~p~dp, \cr &E_{12}=\int_{\mathbb{R}^3}\mathcal{K}_{12}\frac{|p|^2}{2m_1}dp, \quad E_{21}=\int_{\mathbb{R}^3}\mathcal{K}_{12}\frac{|p|^2}{2m_2}dp
\end{split}
\end{align}
and derive the following macroscopic equations for the quantum BGK model \eqref{BGK_quantum}. 
\begin{theorem}
Assume $\nu_{12} n_2 = \nu_{21} n_1$. Let $(f_1,f_2)$ be a solution to \eqref{BGK_quantum}, then we obtain the following formal conservation laws
\begin{align}
\begin{split}
\label{macrosequ1}
&\partial_t n_1 + \frac{1}{m_1} \nabla_x \cdot P_1
=0, \quad
\partial_t n_2 + \frac{1}{m_2} \nabla_x \cdot P_2
=0,
\\
& \partial_t P_1 +\frac{1}{m_1} \nabla_x \cdot \int p \otimes p f_1 dp
 = \nu_{12} n_2 (P_{12} - P_1)
\\
& \partial_t P_2+\frac{1}{m_2} \nabla_x \cdot \int p \otimes p f_1 dp
 = \nu_{21}n_1(P_{21}-P_2)
\\
&\partial_t E_1 + \frac{1}{2m_1^2}\nabla_x \cdot \int |p|^2p f_1 dp
= \nu_{12}n_2(E_{12}-E_1), \cr
&\partial_t E_2 + \frac{1}{2m_2^2}\nabla_x \cdot \int |p|^2p f_2 dp
= \nu_{21}n_1(E_{21}-E_2),
\end{split}
\end{align}
where the exchange terms of momentum can be computed as
\begin{align}
P_{12}-P_1 = - (P_{21}-P_2) = \frac{N_1 N_2}{N_1+N_2} \left(\frac{P_2}{N_2} - \frac{P_1}{N_1} \right).
\label{ex_mom}
\end{align}
Furthermore, we define the functions $\eta_{\tau}(c)=\int \frac{1}{e^{|p|^2+c}+\tau}$ and $\eta^E_{\tau}(c)=\int \frac{|p|^2}{e^{|p|^2+c}+\tau} dp,$
and obtain for the exchange of energy
\begin{align}
\begin{split}
    \label{ex_en}
    &E_{12}-E_1 = -(E_{21}-E_2)\\&=\frac{1}{2} \frac{N_1 |P_1+P_2|^2}{(N_1 + N_2)^2} + \frac{(E_1+E_2)- \frac{1}{2} \frac{|P_1+P_2|^2}{ N_1+N_2}}{m_1^{3/2} \eta^E_{\tau}(c_{12}) + m_2^{3/2} \eta^E_{\tau}(c_{21})} m_1^{3/2} \eta^E_{\tau}(c_{12})- E_1. 
    \end{split}
\end{align}
\end{theorem}
\begin{proof}
We multiply the  first  equation  of  \eqref{BGK_quantum}  by  $(1,p, \frac{|p|^2}{2 m_1})$,  and the second one by $(1, p, \frac{|p|^2}{2 m_2}).$  Then we integrate them with respect to the momentum $p$ to obtain \eqref{macrosequ1} after a straight-forward computation on the left-hand side. 

The exchange of momentum can be computed as follows. 
Computing the integral in the definition of $P_{12}$ and $P_{21}$ in \eqref{def12}, we observe that $P_{12}=bN_1$ and $P_{21}=bN_2$. Substituting the quantity of \textcolor{black}{$b$ in \eqref{abform}}, we obtain
\begin{align*}
\frac{P_{12}}{N_1} - \frac{P_1}{N_1}&=   \frac{P_1+P_2}{N_1+N_2}-\frac{P_1}{N_1} = \frac{N_2}{N_1+N_2}\left(\frac{P_2}{N_2} - \frac{P_1}{N_1}\right)
\end{align*}
and
\begin{align*}
\frac{P_{21}}{N_2} - \frac{P_2}{N_2}&=  \frac{P_1+P_2}{N_1+N_2}= \frac{N_1}{N_1+N_2}\left(\frac{P_1}{N_1} - \frac{P_2}{N_2}\right).
\end{align*}
Similar to section 2 in \cite{BaeYun}, we get 
\begin{align*}
E_{12}- \frac{1}{2} \frac{|P_{12}|^2}{ N_1} = \frac{1}{2} a^{-5/2} m_1^{3/2} \eta^E_{\tau}(c_{12}).
\end{align*}
We can replace \textcolor{black}{$a^{-5/2}$ with the formula \eqref{abform}}
and obtain
\begin{align*}
&E_{12}-E_1 =  -(E_{21}-E_2) \\ &= \frac{1}{2} \textcolor{black}{\frac{N_1 |P_1+P_2|^2}{(N_1 +N_2)^2}} + \frac{(E_1+E_2)- \frac{1}{2} \frac{|P_1+P_2|^2}{N_1+N_2}}{m_1^{3/2} \eta^E_{\tau}(c_{12}) + m_2^{3/2} \eta^E_{\tau'}(c_{21})} m_1^{3/2} \eta^E_{\tau}(c_{12})- E_1. 
\end{align*}
\end{proof}

\begin{remark}
For later purposes, we remark that in the space-homogeneous case the system of equations reduces to
\begin{align}
\label{macrosequ2}
\begin{split}
\partial_t n_1 &=0, \quad \hspace{2.5cm}
\partial_t n_2 =0, \\
\partial_t P_1&= \nu_{12}n_2(P_{12} - P_1), \quad
\partial_t P_2= \nu_{21}n_1(P_{21}-P_2), \\
\partial_t E_1&= \nu_{12}n_2(E_{12}-E_1), \quad
\partial_t E_2 = \nu_{21}n_1(E_{21}-E_2).
\end{split}
\end{align}
\end{remark}

\begin{remark}
\label{remark:classic-classic}
In the classical case $(\tau=\tau'=0)$ we get 
\begin{align}\label{nnrelation}
 \frac{n_1}{n_2}= \frac{m_1^{3/2}\eta_{0}(c_{12})}{m_2^{3/2}\eta_{0}(c_{21})}=\frac{m_1^{3/2}\int_{\mathbb{R}^3}\frac{1}{e^{|p|^2+c_{12}}}dp}{m_2^{3/2}\int_{\mathbb{R}^3}\frac{1}{e^{|p|^2+c_{21}}}dp}=\frac{m_1^{3/2}}{m_2^{3/2}} \frac{e^{-c_{12}}}{e^{-c_{21}}} 
\end{align}
by computing the integrals explicitly in the relationship \eqref{relationships}.
Using this, we can calculate
\begin{align*}
    \frac{m_1^{3/2} \eta_0^E(c_{12})}{m_1^{3/2} \eta_0^E(c_{12}) + m_2^{3/2} \eta^E_0(c_{21})} = \frac{m_1^{3/2} e^{-c_{12}}}{m_1^{3/2} e^{-c_{12}} + m_2^{3/2} e^{-c_{21}}} 
    =\frac{n_1}{n_1+n_2}
\end{align*}
and obtain
\begin{align*}
    E_{12}-&E_1 
    = \frac{n_1 n_2}{n_1+n_2} \left( \frac{E_2}{n_2} - \frac{E_1}{n_1} + \frac{m_1-m_2}{(N_1+N_2)^2} \frac{1}{2} |P_1+P_2|^2 \right) \\ &= \frac{n_1 n_2}{n_1+n_2} \Bigg( \frac{E_2}{n_2} - \frac{1}{2} \frac{|P_2|^2}{n_2 N_2}- \frac{E_1}{n_1} + \frac{|P_1|^2}{n_1 N_1} + m_1 m_2 \frac{n_1 N_1 + 2 n_1 N_2 + n_2 N_2}{( N_1 +  N_2)^2} \frac{1}{2} \frac{|P_2|^2}{ N_2^2} \\&- m_1 m_2 \frac{n_1 N_1 + 2  N_1 n_2 + n_2 N_2}{( N_1 + N_2)^2} \frac{1}{2} \frac{|P_1|^2}{ N_1^2} + m_1 m_2 \frac{(m_1-m_2) n_1 n_2}{( N_1 + N_2)^2} \frac{P_1}{ N_1} \cdot \frac{P_2}{ N_2} \Bigg)\\ &= \frac{n_1 n_2}{n_1+n_2} \Bigg( \frac{E_2}{n_2} - \frac{1}{2} \frac{|P_2|^2}{n_2 N_2}- \frac{E_1}{n_1} + \frac{|P_1|^2}{n_1 N_1} + m_1 m_2 \frac{n_1 N_1+n_2 N_2}{( N_1 + N_2)^2} \frac{1}{2} \left( \frac{|P_2|^2}{ N_2^2} - \frac{|P_1|^2}{ N_1^2} \right)\\& + \left( \frac{P_2}{ N_2} - \frac{P_1}{ N_1} \right) \cdot \left( \frac{P_1}{ n_1}  + \frac{P_2}{n_2} \right)\Bigg).
\end{align*}
\end{remark}
\subsection{Convergence rate for the velocities and kinetic temperatures in the space-homo\-ge\-neous case}\label{app:velocities}
In equilibrium, both distribution functions $f_1$ and $f_2$ will finally share the same velocity and the same temperature. For the momentum, we can prove the following exponential convergence rate.
\begin{theorem}\label{th:estimate_vel}
Assume $\nu_{12} n_2= \nu_{21} n_1$ and $ \nu_{12}, \nu_{21}$ independent of $t$. In the space-homogeneous case of \eqref{BGK_quantum}, we have the following convergence rate for the momentum 
\begin{equation} \label{eq:quantum-decay-momentum}
\frac{P_1(t)}{N_1}-\frac{P_2(t)}{ N_2} = e^{-\frac{\textcolor{black}{\nu_{12} n_2 N_2 + \nu_{21} n_1 N_1} }{N_1+N_2}t} \left(\frac{P_1(0)}{ N_1} - \frac{P_2(0)}{ N_2}\right).
\end{equation}
\end{theorem}
\begin{proof}
We start with calculating 
\begin{align*}
\partial_t \left(\frac{P_1}{ N_1} \right)  = \frac{1}{ N_1} \partial_t P_1 
= \textcolor{black}{\nu_{12} n_2}\frac{1}{ N_1}(P_{12} - P_1). 
\end{align*}
The first equality holds true because $N_1= m_1 n_1$ is constant in the space-homogeneous case \eqref{macrosequ2}. Then, we inserted the macroscopic equation for the time evolution of $P_1$ from \eqref{macrosequ2}. 
Using the expression \eqref{ex_mom} for the exchange of momentum leads to
\begin{align}
\partial_t \left( \frac{P_1}{ N_1} \right) = \textcolor{black}{\nu_{12} n_2} \frac{ N_2}{ N_1+ N_2} \left( \frac{P_2}{ N_2} - \frac{P_1}{ N_1} \right).
\label{momentum}
\end{align}
In a similar way, we can compute
\begin{align*}
\partial_t \left( \frac{P_2}{ N_2} \right) = \textcolor{black}{\nu_{21} n_1} \frac{N_1}{N_1+ N_2} \left( \frac{P_1}{ N_1} - \frac{P_2}{ N_2} \right).
\end{align*}
Substracting both equations yields 
\begin{align*}
\partial_t \left( \frac{P_1}{N_1} - \frac{P_2}{ N_2} \right)= - \frac{\textcolor{black}{\nu_{12} n_2 N_2 + \nu_{21} n_1 N_1} }{N_1+N_2}\left(\frac{P_1}{ N_1} - \frac{P_2}{ N_2} \right),
\end{align*}
and it gives the result \eqref{eq:quantum-decay-momentum}.
\end{proof}

\begin{remark}
Using the relationship $P_j= N_j b_j$ from \eqref{ai}, one can equivalently write
\begin{align*}
b_1 -b_2 = e^{-\frac{\nu_{12} n_1 N_2 + \nu_{21} n_2 N_1}{N_1+N_2}t} (b_1(0) - b_2(0)).
\end{align*}
\end{remark}
We continue with deriving the convergence rates of the quantities $\frac{E_1}{ n_1}- \frac{1}{2} \frac{|P_1|^2}{n_1 N_1}$ and $\frac{E_2}{ n_2}- \frac{1}{2} \frac{|P_2|^2}{n_2 N_2}$. Only in the classical case, these quantities correspond to the (physical) temperature. Whereas in the quantum case, we need to specify them as kinetic temperatures.

  \begin{theorem}\label{th:estimate_temp}
Let $\nu_{12} n_2 =\nu_{21} n_1 =: \tilde{\nu}$ and $\tilde{\nu}$ be independent of $t$.
In the space-homogeneous case of \eqref{BGK_quantum}, we have 
\begin{align} \label{eq:quantum-decay-kinetic-T}
\begin{split}
&\left( \frac{E_1(t)}{n_1}- \frac{1}{2} \frac{|P_1(t)|^2}{n_1 N_1} \right) - \left(\frac{E_2(t)}{ n_2} - \frac{1}{2} \frac{|P_2(t)|^2}{\textcolor{black}{n_2 N_2} }\right) \\&= e^{-\tilde{\nu} t} \left( \left( \frac{E_1(0)}{ n_1} - \frac{1}{2} \frac{|P_1(0)|^2}{n_1 N_1}\right) - \left( \frac{E_2(0)}{ n_2} - \frac{1}{2} \frac{|P_2(0)|^2}{n_2 N_2} \right) \right) \\
&+ \frac{1}{2} m_1 m_2 \frac{n_2 N_2- n_1 N_1}{(N_1+N_2)^2} e^{-\tilde{\nu}t} (1- e^{-\tilde{\nu}t}) \bigg| \frac{P_2(0)}{ N_2} - \frac{P_1(0)}{ N_1} \bigg|^2  \\
&+\tilde{\nu}\left(E_1(0)+E_2(0)- \frac{1}{2} \frac{|P_1(0)+P_2(0)|^2}{ N_1+ N_2} \right) e^{-\tilde{\nu} t}  \int_0^t e^{\tilde{\nu}s }\left[\frac{\frac{m_1^{3/2} \eta^E_{\tau}(c_{12}(s))}{ n_1}- \frac{m_2^{3/2} \eta^E_{\tau'}(c_{21}(s))}{ n_2}}{m_1^{3/2} \eta^E_{\tau}(c_{12}) + m_2^{3/2} \eta^E_{\tau'}(c_{21})} \right] ds. 
\end{split}
\end{align}

\end{theorem}
\begin{proof}
We note that $n_1$ and $N_1$ are constants in the space-homogeneous case. Using \eqref{macrosequ2} and \eqref{momentum} gives 
\begin{align*}
\partial_t \left( \frac{E_1}{ n_1}- \frac{1}{2} \frac{|P_1|^2}{n_1 N_1} \right) &=\frac{1}{n_1} \partial_t \left( E_1 \right) - \frac{P_1}{ n_1} \partial_t \left( \frac{P_1}{ N_1} \right) \\&=\tilde{\nu} \left(\frac{E_{12}}{ n_1}- \frac{E_1}{ n_1} - \frac{P_1}{ n_1} \frac{ N_2}{N_1 +N_2} \left( \frac{P_2}{N_2} - \frac{P_1}{ N_1} \right) \right).
\end{align*}
Applying \eqref{ex_en}, it follows
\begin{align*}
\begin{split}
\partial_t \left(\frac{E_1}{ n_1}- \frac{1}{2} \frac{|P_1|^2}{n_1 N_1} \right) =\tilde{\nu}\textcolor{black}{\bigg[}\frac{1}{2} \frac{m_1 |P_1+P_2|^2}{( N_1 +N_2)^2} + \frac{(E_1+E_2)- \frac{1}{2} \frac{|P_1+P_2|^2}{N_1+N_2}}{m_1^{3/2} \eta^E_{\tau}(c_{12}) + m_2^{3/2} \eta^E_{\tau'}(c_{21})} \frac{m_1^{3/2} \eta^E_{\tau}(c_{12})}{ n_1}\\- \frac{E_1}{ n_1} -  \frac{P_1}{ n_1} \frac{ N_2}{N_1 + N_2} \left( \frac{P_2}{N_2} - \frac{P_1}{N_1} \right)\textcolor{black}{\bigg]}.
\end{split}
\end{align*}
Analogously for species 2 we obtain
\begin{align*}
\partial_t \left(\frac{E_2}{ n_2}- \frac{1}{2} \frac{|P_2|^2}{n_2 N_2} \right) =\tilde{\nu} \textcolor{black}{\bigg[}\frac{1}{2} \frac{m_2|P_1+P_2|^2 }{(N_1 +N_2)^2} + \frac{(E_1+E_2)- \frac{1}{2} \frac{ |P_1+P_2|^2}{N_1+N_2}}{m_1^{3/2} \eta^E_{\tau}(c_{12}) + m_2^{3/2} \eta^E_{\tau'}(c_{21})} \frac{m_2^{3/2} \eta^E_{\tau}(c_{21})}{ n_2}\\- \frac{E_2}{ n_2} -  \frac{P_2}{n_2} \frac{N_1}{ N_1+ N_2} \left(\frac{P_1}{N_1} - \frac{P_2}{N_2}\right) \textcolor{black}{\bigg]}.
\end{align*}
Subtracting both equations, we get
\begin{align}\label{I+II}
\begin{split}
\partial_t \left( \left(\frac{E_1}{ n_1}- \frac{1}{2} \frac{|P_1|^2}{n_1 N_1}\right) -\left(\frac{E_2}{ n_2}-  \frac{|P_2|^2}{n_2 N_2} \right) \right)  &=
\tilde{\nu}(I+II),
\end{split}
\end{align}
where
\begin{align*}
\begin{split}
I &:= \frac{E_2}{ n_2}-\frac{E_1}{ n_1}+
\frac{\frac{1}{2}(m_1-m_2) n_1 N_1 +  N_2 (N_1+N_2)}{(N_1+N_2)^2} \frac{|P_1|^2}{n_1 N_1} \cr 
&+  \frac{\frac{1}{2}(m_1-m_2) n_2 N_2-  N_1 (N_1+N_2)}{(N_1+N_2)^2} \frac{|P_2|^2}{n_2 N_2} +  \frac{ (n_1 N_1-n_2 N_2)}{(N_1 +N_2)^2 n_1 n_2} P_1 \cdot P_2, \cr 
&= -\left(\left(\frac{E_1}{ n_1} -\frac{|P_1|^2}{n_1 N_1}\right) - \left(\frac{E_2}{ n_2} -  \frac{|P_2|^2}{n_2 N_2}\right)\right) +  \frac{1}{2} m_1 m_2 \frac{n_2 N_2-n_1 N_1}{(N_1 + N_2)^2} \bigg|\frac{P_1}{N_1}- \frac{P_2}{N_2}\bigg|^2,
\end{split}
\end{align*}
and
\begin{align*}
\begin{split}
II &:= \frac{(E_1+E_2)- \frac{1}{2} \frac{|P_1+P_2|^2}{ N_1+N_2}}{m_1^{3/2} \eta^E_{\tau}(c_{12})+m_2^{3/2}\eta^E_{\tau'}(c_{21})} \left[\frac{m_1^{3/2} \eta^E_{\tau}(c_{12})}{ n_1}-\frac{m_2^{3/2} \eta^E_{\tau'}(c_{21})}{ n_2} \right].
\end{split}
\end{align*}
\textcolor{black}{
Then, Duhamel's formula to \eqref{I+II} gives the result. }

\end{proof}
\begin{remark} \label{remark:classic-classic-2}
In the classical case, using the relation \eqref{nnrelation} in remark \ref{remark:classic-classic}, once we compute the following quantity
\begin{align*}
    \left[\frac{\frac{m_1^{3/2} \eta^E_{\tau}(c_{12})}{ n_1}- \frac{m_2^{3/2} \eta^E_{\tau'}(c_{21})}{ n_2}}{m_1^{3/2} \eta^E_{\tau}(c_{12}) + m_2^{3/2} \eta^E_{\tau'}(c_{21})} \right],
\end{align*}
for $\tau=\tau'=0$, then by an explicit computation of $\eta^E_{\tau}$, we obtain 
\begin{align*}
    \frac{\frac{m_1^{3/2}}{n_1} \frac{3}{2} e^{-c_{12}} \pi^{3/2} - \frac{m_2^{3/2}}{n_2} \frac{3}{2} e^{-c_{21}} \pi^{3/2}}{m_1^{3/2} \frac{3}{2} e^{-c_{12}} \pi^{3/2} + m_2^{3/2} \frac{3}{2} e^{-c_{21}} \pi^{3/2}} = \frac{\frac{m_1^{3/2}}{n_1} e^{-c_{12}} - \frac{m_2^{3/2}}{n_2} e^{-c_{21}} }{ m_1^{3/2} e^{-c_{12}} + m_2^{3/2} e^{-c_{21}}}=0.
\end{align*}
So in the classical-classical case, the last term of \eqref{eq:quantum-decay-kinetic-T} vanishes, and we obtain the temperature convergence rate
\begin{align*}
   &\left( \frac{E_1\textcolor{black}{(t)}}{ n_1}- \frac{1}{2} \frac{|P_1\textcolor{black}{(t)}|^2}{n_1 N_1} \right) - \left(\frac{E_2\textcolor{black}{(t)}}{ n_2} - \frac{1}{2} \frac{|P_2\textcolor{black}{(t)}|^2}{n_2 N_2 }\right) \\&= e^{-\tilde{\nu}t} \left( \left( \frac{E_1(0)}{ n_1} - \frac{1}{2} \frac{|P_1(0)|^2}{n_1 N_1}\right) - \left( \frac{E_2(0)}{ n_2} - \frac{1}{2} \frac{|P_2(0)|^2}{n_2 N_2} \right) \right) \\&+ \frac{1}{2} m_1 m_2 \frac{n_2 N_2- n_1 N_1}{(N_1+N_2)^2} e^{-\tilde{\nu}t} (1- e^{-\tilde{\nu}t}) \bigg| \frac{P_2(0)}{N_2} - \frac{P_1(0)}{ N_1} \bigg|^2.  
\end{align*}
\end{remark}
In the remainder of the paper, we establish a numerical scheme which fulfills above physical properties on a discrete level. 

\section{Numerical scheme}


\subsection{Time discretization}
\label{sec:time}

Let $k,j=1,2$ and $k\neq j$. We write \eqref{BGK_quantum} as 
\begin{align}\label{QBGKdt}
    \partial_t f_k + \mathcal{T}_k(f_k) = \mathcal{R}_{k}(f_k,f_j)
\end{align}
with the combined relaxation operator  
\begin{align*}
\mathcal{R}_{k}(f_k,f_j) = \mathcal{R}_{kk} + \mathcal{R}_{kj} =  \nu_{kk} n_k\left(\E_{kk}-f_k\right) + \nu_{kj}n_j\left(\E_{kj}-f_k\right)
\end{align*}
and the transport operator 
\begin{align*}
\mathcal{T}_k(f_k) = \frac{p}{m_k}\cdot \nabla_x f_k. 
\end{align*}
In the following, for simplicity, we assume that the collision frequencies $\tilde{\nu}_{kk}:=\nu_{kk} n_k$ and $\tilde{\nu}_{kj}:=\nu_{kj} n_k$ are constant in $x$ and $t$. But an extension to an $x$ and $t$ dependence of the collision frequency would also be possible. 
Large collision frequencies result in a stiff relaxation operator such that an implicit time discretization for the relaxation part is a convenient choice. We pursue implicit-explicit (IMEX) schemes where $\mathcal{R}_{k}$ is treated implicitly and $\mathcal{T}_i$ is treated explicitly.  

Given $t_\ell = \ell \Delta t$ for $\ell\in \mathbb{N}_0$, a simple update of $f_k^\ell \approx f_k(x,p,t_\ell)$ from $t_{\ell}$ to $t_{\ell+1}$ uses the approximation 
\begin{align*}
\mathcal{R}_{k}(f_k^{\ell+1},f_j^{\ell+1}) \approx  \tilde{\nu}_{kk} \left(\E_{kk}^{\ell+1}-f_k^{\ell+1}\right) + \tilde{\nu}_{kj} \left(\E_{kj}^{\ell+1}-f_k^{\ell+1}\right),
\end{align*}
where $\E_{kk}^{\ell+1}$ and $\E_{kj}^{\ell+1}$ are discrete target functions depending on $f_k^{\ell+1}$ and $f_j^{\ell+1}$ via the solution of a convex minimization problem that is inspired by the work in \cite{velocity-dependent}. We discuss it in Section \ref{sec:generaltime}. By this procedure, $\E_{kk}$ and $\E_{kj}$ are evaluated exactly at the next time step  (up to numerical tolerances) which results in the preservation of conservation properties, and the first-order version inherits additional properties from the continuum model.  

\subsubsection{First-order splitting} \label{subsec:firstordersplit}
We split the relaxation and transport operators in \textcolor{black}{\eqref{QBGKdt}}. 
\paragraph{Relaxation.} We perform the relaxation step in each spatial cell by a backward Euler method
\begin{align} \label{eq:relax}
\frac{f_k^{\ast} -f_k^\ell}{\Delta t} = \mathcal{R}_{k}(f_k^{\ast},f_j^{\ast}),
\end{align}
which can be rewritten into the convex combination
\begin{align} \label{eq:update_split}
    f_k^{\ast} = d_k f_k^{\ell} + d_k \Delta t ( \tilde{\nu}_{kk}\E_{kk}^{\ast} + \tilde{\nu}_{kj} \E_{kj}^{\ast}),
\end{align}
with 
\begin{align*} 
    d_k = \frac{1}{1+ \Delta t (\tilde{\nu}_{kk} + \tilde{\nu}_{kj}) }.
\end{align*}
The equation \eqref{eq:update_split} represents an explicit update formula for $f_k^{\ast}$
provided that $\E_{kk}^{\ast}$ and $\E_{kj}^{\ast}$ can be expressed as functions of $f_k^\ell$. In Section \ref{sec:generaltime} we show how to determine $\E_{kk}^{\ast}$ and $\E_{kj}^{\ast}$ in a structure-preserving way.
\paragraph{Transport.} We compute the transport in $x$ for $f_k^{\ell+1}$ by a forward Euler method with initial data $f_k^{\ast}$:
\begin{align} \label{eq:transport_x}
\frac{f_k^{\ell+1} -f_k^{\ast}}{\Delta t}+ \mathcal{T}_k(f_k^{\ast}) = 0.
\end{align}
Details on the numerical approximation of $\mathcal{T}_k$ are presented in section \ref{sec:space}.

\subsubsection{Second-order IMEX Runge-Kutta} \label{subsec:secondorderIMEX}

We use the following Butcher tableaux \cite{ARS97} for a second-order approach
\begin{align*}
\renewcommand\arraystretch{1.2}
\begin{array}
{c|ccc}
0\\
\gamma &0& \gamma \\
1 &0&1-\gamma&\gamma \\
\hline
& 0&1-\gamma&\gamma
\end{array}
\hspace*{2cm}
\renewcommand\arraystretch{1.2}
\begin{array}
{c|ccc}
0\\
\gamma & \gamma \\
1 &\delta &1-\delta&0 \\
\hline
& \delta &1-\delta&0
\end{array}
\end{align*}
with 
\begin{equation*}
    \gamma = 1- \frac{\sqrt{2}}{2} 
    \quad \text{and} \quad
    \delta = 1-\frac{1}{2 \gamma}.
\end{equation*}
The left table applies to the relaxation part, and the right table applies to the transport terms. 
This IMEX Runge-Kutta scheme is L-stable and globally stiffly accurate.
\\
Applying this method to \textcolor{black}{\eqref{eq:relax} and \eqref{eq:transport_x}} and using the constants 
\begin{align}
    d_k = \frac{1}{1+ \gamma \Delta t (\tilde{\nu}_{kk} + \tilde{\nu}_{kj}) },
    \label{di}
\end{align}
we can write the stages in the scheme as convex combination of three terms
\begin{subequations}  \label{eq:update_ARS}
\begin{alignat}{3}
f_k^{(1)} &= d_k G_k^{(1)}  &&+ d_k\gamma \Delta t \,  \tilde{\nu}_{kk} \E_{kk}^{(1)} &&+  d_k  \gamma \Delta t \, \tilde{\nu}_{kj}\E_{kj}^{(1)}  \\
f_k^{(2)} &= d_k G_k^{(2)} &&+ d_k \gamma \Delta t \,  \tilde{\nu}_{kk}\E_{kk}^{(2)} &&+ d_k \gamma \Delta t \,  \tilde{\nu}_{kj}\E_{kj}^{(2)}, \\
f_k^{\ell+1} &= f_k^{(2)}
\end{alignat}
\end{subequations}
where
\begin{subequations} 
\begin{align}
G_k^{(1)}  &= f_k^\ell - \Delta t \, \gamma\, \mathcal{T}_k(f_k^\ell) \\
G_k^{(2)} &=  f_k^\ell - \Delta t \, \delta\, \mathcal{T}_k(f_k^\ell)-\Delta t \,(1-\delta) \mathcal{T}_k(f_k^{(1)}) + \Delta t \,(1-\gamma) \mathcal{R}_{k}(f_k^{(1)},f_j^{(1)})
\end{align}
\end{subequations}
depend on known data.  For each stage, we have to determine the corresponding values of the target functions in order to update the distribution functions. In the following section, we explain how this can be achieved.

\subsubsection{General implicit solver} \label{sec:generaltime}
We write the implicit updates in \eqref{eq:update_split} and \eqref{eq:update_ARS} in a generic steady state form
\begin{align} \label{eq:update_general}
\psi_k  = d_k G_k  + d_k \gamma\Delta t  (\tilde{\nu}_{kk} \E_{kk}  + \tilde{\nu}_{kj} \E_{kj} ). 
\end{align}
The functions $\E_{kk}$ and $\E_{kj}$ are the unique target functions associated to $\psi_k$, 
\begin{equation*}
    d_k = \frac{1}{1+\gamma\Delta t (\tilde{\nu}_{kk}+\tilde{\nu}_{kj})},
\end{equation*}
and $G_k$ is a known function.  We want to express $\E_{kk} $ and $\E_{kj} $ as functions of $G_k$ and $G_j$ so that \eqref{eq:update_general} is an explicit update formula for $\psi_k$.  
In Section \ref{sec:model}, the existence and uniqueness of $\E_{kk}$ and $\E_{kj}$ are \textcolor{black}{presented} by algebraic considerations. In order to determine their values numerically, we follow a different approach \textcolor{black}{in Appendix} \ref{sec2}. We introduce the notation
\begin{align*}
    \mathbf{p}_k(p) := (1,p, \frac{|p|^2}{2m_k})^\top, \quad k=1,2.
\end{align*}
Applying the conservation properties \eqref{conserv 1} and \eqref{conserv 2} to \eqref{eq:update_general} leads to 
\begin{align*}
\begin{split}
\int \tilde{\nu}_{11} \E_{11}  \,\mbp_1 dp +
\int \tilde{\nu}_{22} \E_{22}  \,\mbp_2 dp +
\int \tilde{\nu}_{12} \E_{12}  \,\mbp_1 dp + 
\int \tilde{\nu}_{21} \E_{21}  \,\mbp_2 dp \end{split} \nonumber \\ 
\begin{split} \overset{\eqref{conserv 1},\eqref{conserv 2}}{=}
\int  \tilde{\nu}_{11} \psi_1  \,\mbp_1 dp +
\int  \tilde{\nu}_{22} \psi_2  \,\mbp_2 dp +
\int  \tilde{\nu}_{12} \psi_1  \,\mbp_1 dp +
\int  \tilde{\nu}_{21} \psi_2  \,\mbp_2 dp 
\end{split} \nonumber \\ 
\begin{split}
&\hspace{0.25cm}\overset{(\ref{eq:update_general})}{=}
  \int  \tilde{\nu}_{11}   d_1 \left[ G_1  +  \Delta t \, \gamma  \tilde{\nu}_{11} \E_{11}  + \Delta t \, \gamma  \tilde{\nu}_{12} \E_{12}  \right] \,\mbp_1 dp  
\\& \hspace*{1.2cm}+  \int \tilde{\nu}_{22}  d_2 \left[ G_2  + \Delta t  \, \gamma  \tilde{\nu}_{22}  \E_{22}  + \Delta t \, \gamma \tilde{\nu}_{21}  \E_{21}  \right]   \,\mbp_2 dp  \\
&\hspace*{1.2cm}  
 + \int  \tilde{\nu}_{12}   d_1 \left[ G_1  +  \Delta t \, \gamma  \tilde{\nu}_{11} \E_{11}  + \Delta t \, \gamma  \tilde{\nu}_{12} \E_{12}  \right] \,\mbp_1 dp  
\\&\hspace*{1.2cm}+  \int \tilde{\nu}_{21}  d_2 \left[ G_2  + \Delta t  \, \gamma  \tilde{\nu}_{22}  \E_{22}  + \Delta t \, \gamma \tilde{\nu}_{21}  \E_{21}  \right]   \,\mbp_2 dp  
\end{split}
\end{align*}
Sorting terms yields the following moment equations
\begin{align*} 
\begin{split}
\int d_1  \left( \tilde{\nu}_{11}\E_{11}+ \tilde{\nu}_{12}\E_{12}\right) \mbp_1 dp
+  \int d_2   \left( \tilde{\nu}_{21}\E_{21}+ \tilde{\nu}_{22}\E_{22}\right) \mbp_2 dp \\
= \int  d_1(\tilde{\nu}_{11}+\tilde{\nu}_{12})  G_1  \mbp_1 + \int  d_2 (\tilde{\nu}_{21}+\tilde{\nu}_{22})   G_{2}  \mbp_2 dp
\end{split}
\end{align*}
which provide a set of constraints to determine  $\E_{kk}$ and $\E_{kj}$ from the given data $G_k$ and $G_j$. These  represent first-order optimality conditions associated to the minimization of the convex  potential function 
\begin{align}\label{eq:potentialfunction}
\begin{split}
 \varphi_{\rm tot}(\alpha_1,\alpha_2,\alpha) &=
\int  \left[ d_1 \tilde{\nu}_{11} w(\E_{11}) + d_2 \tilde{\nu}_{22} w(\E_{22}) +  d_1 \tilde{\nu}_{12} w(\E_{12}) + d_2 \tilde{\nu}_{21} w(\E_{21}) \right] dp \\ 
&+ \quad \mu_1 \cdot \alpha_1 + \mu_2 \cdot \alpha_2 + \mu \cdot \alpha, 
\end{split}
\end{align} 
where $\alpha_k=(\alpha^0_k, \alpha^1_k, \alpha^2_k)$ and $\alpha = (\alpha_{12}^0,\alpha_{21}^0,\alpha^1,\alpha^2)^\top$;
the auxiliary function reads
\textcolor{black}{
\begin{align*}
w(\E_{kj,\tau_{k}}) 
=
\begin{cases}
	-\E_{kj,0} \qquad  &\text{for } \tau_{k} = 0, \\
	 \log(1-\E_{kj,+1})\quad &\text{for } \tau_{k} =  +1, \\
	 -\log(1+\E_{kj,-1})\quad &\text{for } \tau_{k} =  -1.
\end{cases}
\end{align*}}
the given moments are
\begin{align} \label{eq:input_mu_i}
\mu_k = 
\begin{pmatrix}
\mu_{k}^0 \\ \mu_{k}^1 \\ \mu_{k}^2
\end{pmatrix} =
\int d_k \tilde{\nu}_{kk}  G_k \mbp_k dp,
\end{align}
for $k=1,2$; and
\begin{align} \label{eq:input_mu}
\mu = 
\begin{pmatrix}
\mu_{12}^0 \\ \mu_{21}^0 \\ \mu^1 \\ \mu^2
\end{pmatrix} &= \int \left[ 
\begin{pmatrix}
1 \\ 0 \\ p\\ \frac{|p|^2}{2m_1} 
\end{pmatrix}
  d_1 \tilde{\nu}_{12}  G_1 +
\begin{pmatrix}
0 \\ 1 \\ p \\ \frac{|p|^2}{2m_2} 
\end{pmatrix}
 d_2 \tilde{\nu}_{21}  G_2  \right] dp.
\end{align}
The minimization problem can be decoupled as follows:
\begin{proposition}
The components of the minimizer of (\ref{eq:potentialfunction}) can be found by minimizing the following three convex potential functions independently:
  \begin{align}
     \varphi_k(\alpha_k) &= \int d_k \tilde{\nu}_{kk} \, w(\E_{kk}) dp + \mu_k\cdot \alpha_k \quad \text{for} \quad k=1,2 \quad \text{and} \label{eq:potentialfunction_single} \\
     \varphi(\alpha) &= \int  \left[ d_1 \tilde{\nu}_{12} w(\E_{12}) + d_2 \tilde{\nu}_{21} w(\E_{21}) \right] dp  + \mu \cdot \alpha \label{eq:potentialfunction_mixed}
 \end{align}
and the minimum of \eqref{eq:potentialfunction} is the sum of their minima.
\end{proposition}

\begin{proof}
The statement is trivial because  $\varphi_{\rm tot}(\alpha_1,\alpha_2,\alpha) = \varphi_1(\alpha_1) + \varphi_2(\alpha_2) + \varphi_1(\alpha)$.
\end{proof}

The minimum of each potential function in \eqref{eq:potentialfunction_single} and \eqref{eq:potentialfunction_mixed} is found using Newton's method for convex optimization. More details are given in Section \ref{sec:velocity}. 

Actually, we can link these potential functions to dual problems when we reformulate the modelling problem by using Lagrange functionals. For intra-species interactions, the Lagrange functional reads 
\begin{align} \label{eq:Lagrange-single}
    L_{k} (g,\lambda) = \int h_{\tau_k}(g) dp - \lambda \cdot \int \mbp_k (g-f_k) dp
\end{align}
using
$h_{\tau_k}(g)$ given by \eqref{entropy_h}. 
The first integral in \eqref{eq:Lagrange-single} is the entropy functional; the other integrals describe the conservation properties as constraints. \textcolor{black}{Substituting $g=\frac{1}{e^{-\lambda \cdot \mbp_k} +\tau_k}$ to \eqref{eq:Lagrange-single},} the Lagrange multipliers $\lambda$ solve the dual problem
\begin{align} \label{eq:dual-intra}
    \alpha_k = \operatorname*{argmin}_{\lambda \in \Lambda_k} \int w(\E_{kk}(\lambda)) dp + \lambda \cdot \int \mbp_k f_k dp
\end{align}
where $\Lambda_k=\{\lambda \in \mathbb{R}^5\,|\, \int \E_{kk}(\lambda) (1+|p|^2)dp < \infty \}$.
Analogously, we can formulate the dual problem for inter-species interactions:
\begin{align} \label{eq:dual-inter}
\begin{split}
    (\alpha_{12},\alpha_{21})= \operatorname*{argmin}_{(\lambda_{12},\lambda_{21}) \in \Lambda_{12}} \Bigg\{ \int& w(\E_{12}(\lambda)) +  w(\E_{21}(\lambda)) dp + \lambda_{12}^0  \int  f_1 dp + \lambda_{21}^0  \int  f_2 dp \\
    &+ \lambda^1 \cdot  \int p (f_1+f_2) dp + \lambda^2 \int |p|^2 \left( \frac{1}{2m_1} f_1 + \frac{1}{2m_2}f_2\right) dp\Bigg\}
\end{split}
\end{align}
for \textcolor{black}{$\alpha_{kj}=(\alpha_{kj}^0,\alpha^1,\alpha^2)$ and where $\Lambda_{12}=\{(\lambda_{12}^0,\lambda_{21}^0,\lambda^1,\lambda^2) \in \mathbb{R}^6\,|\, \int \E_{kj}(\lambda_{kj}) (1+|p|^2)dp < \infty \text{ for } k,j=1,2;k\neq j\}$.}
We recognize the close relationship of \eqref{eq:potentialfunction_single} with \eqref{eq:dual-intra}, respective of \eqref{eq:potentialfunction_mixed} with \eqref{eq:dual-inter}. The dual problems have unique solutions according to Appendix \ref{sec2}. This is inherited to the potential functions because $d_k\tilde{\nu}_{kj}$ is independent of $p$. 

\subsection{Space discretization} \label{sec:space}

We assume a slab geometry, i.e. $\partial_{x^2} f_k = \partial_{x^3} f_k = 0$. So we reduce the physical space dimension to one dimension and set $x := x^1$ while the momentum domain remains three dimensional $(p=(p^1,p^2,p^3))$.  We divide the spatial domain $[x_{\rm min},x_{\rm max}]$ into uniform cells $I_i = [x_i-\frac{\Delta x}{2}, x_i+\frac{\Delta x}{2}]$ for $i \in \{0,\dots,I\}$. 

We employ a second-order finite volume framework using approximate cell-averaged quantities 
\begin{equation*}
    f_{k,i}^\ell \approx \frac{1}{\Delta x}\int_{I_i} f_k(x,p,t^{\ell}) dx.
\end{equation*}
The relaxation operators are approximated to second order by
\begin{equation*}
    \mathcal{R}_{k,i}^\ell 
        = \mathcal{R}_{k}(f_{k,i}^\ell,f_{j,i}^\ell) 
        \approx \frac{1}{\Delta x} \int_{I_i} \mathcal{R}\left(\,f_k(x,p,t^{\ell}),f_j(x,p,t^{\ell}) \,\right) dx.
\end{equation*}
Whereas the transport operator $\mathcal{T}_k$ is discretized with numerical fluxes $\mathscr{F}_{i+\frac{1}{2}}$ by 
\begin{align*} 
\mathcal{T}_k\left( g \right) \approx \mathcal{T}_{i;k}(g) = \frac{1}{\Delta x} \left( \textcolor{black}{\mathscr{F}_{i+\frac{1}{2}}(g)} - \mathscr{F}_{i-\frac{1}{2}}(g) \right)
\end{align*}
for any grid function $g = \{g_{i}\}$. We follow \cite{MieussensStruchtrup2004} and use
\begin{align*} 
    \mathscr{F}_{i+\frac{1}{2}}(g) = \frac{p^1}{2m} \left( g_{i+1} + g_{i} \right) - \frac{\vert p^1 \vert }{2m} \left( g_{i+1} - g_{i} - \phi_{i+\frac{1}{2}}(g) \right)
\end{align*}
where $\phi_{i+\frac{1}{2}}$ is a flux limiter. The choice $\phi_{i+\frac{1}{2}}=0$ leads to a first-order approximation, and a second-order method is provided by 
\begin{align*} 
\phi_{i+\frac{1}{2}}(g) = \operatorname{minmod}\left( (g_{i}-g_{i-1}), (g_{i+1}-g_{i}), (g_{i+2}-g_{i+1})\right)
\end{align*}
where
\begin{align*}
    \operatorname{minmod}(a,b,c) = 
    \begin{cases}
     s \min(|a|,|b|,|c|), \quad & \mathrm{sign}(a) = \mathrm{sign}(b)=\mathrm{sign}(c) =:s,\\
     0, \quad &\text{otherwise}.
    \end{cases}
\end{align*}
We guarantee positivity during a simple forward Euler update of \eqref{eq:transport_x} by enforcing the CFL condition 
\begin{align*} 
   \Delta t < \beta \frac{m\Delta x}{\max |p^1|}
\end{align*}
with $\beta=1$ for the first-order flux and $\beta = \frac{2}{3}$ for the second-order flux.  (See Proposition \ref{prop:positivity_1st_order}.)

\subsection{Properties of the semi-discrete scheme}
\label{sec:discreteproperties}

In this section, we review the positivity preservation, conservation properties, and the entropy behavior of the semi-discrete scheme.

\subsubsection{Positivity of distribution functions}

The first-order time stepping scheme in Section \ref{subsec:firstordersplit} preserves positivity for both first- and second-order numerical fluxes in space; see Proposition \ref{prop:positivity_1st_order}. 
We discuss the positivity for the second-order scheme \ref{subsec:secondorderIMEX} in Proposition \ref{prop:positivity_2nd_order}, and give a sufficient criterion for the space homogeneous case.
Additionally, we show that the upper bound for distribution functions of fermions is preserved by our scheme; see Proposition \ref{prop:fermion-smaller-1}.

\begin{proposition}\label{prop:positivity_1st_order}
The first-order time discretization in Section \ref{subsec:firstordersplit} together with the space discretization described in Section \ref{sec:space} is positivity preserving, provided that
\begin{align*} 
     \Delta t \leq \beta \frac{m_k\Delta x}{\max |p^1|},
 \end{align*}
 with $\beta=1$ and $\beta = \frac{2}{3}$ for the first-order and second-order fluxes, respectively. 
\end{proposition}
\begin{proof}
The proof can be performed analogously to the proof of Proposition 5.1 in \cite{velocity-dependent}.
\end{proof}

Second-order time-stepping makes it more difficult to guarantee positivity. 
Nevertheless, we derive some sufficient conditions on $ \Delta t$ in order to preserve positivity  in the second-order scheme presented in Section \ref{subsec:secondorderIMEX}.

\begin{proposition} \label{prop:positivity_2nd_order}
For the space homogeneous case, the second-order IMEX scheme presented in Section \ref{subsec:secondorderIMEX} is positivity preserving provided that
\begin{align} \label{eq:timerestriction-IMEX}
\Delta t \leq \frac{1}{(1-2\gamma)(\tilde{\nu}_{kk}+\tilde{\nu}_{kj})}
\end{align}
 for $k,j=1,2.$
\end{proposition}
\begin{proof}
The proof can be performed analogously to the proof of proposition 5.2 in \cite{velocity-dependent}.
\end{proof}

For large collision frequencies $\nu_{kj}$, the time step condition \eqref{eq:timerestriction-IMEX} can be restrictive. So one might be interested in enforcing the milder (but still sufficient) local condition
\begin{align}\label{eq:timestep-restriction-local-IMEX}
\Delta t \leq \frac{f_k^{\ell}}{(1-\gamma)  \left[(\tilde{\nu}_{kk}+\tilde{\nu}_{kj})f_k^{(1)} - (\tilde{\nu}_{kk}\E_{kk}^{(1)} + \tilde{\nu}_{kj}\E_{kj}^{(1)}) \right]}.
\end{align}
Large collision frequencies push the numerical kinetic distribution to the corresponding target function. Hence, the denominator in \eqref{eq:timestep-restriction-local-IMEX} becomes large, and the condition is not restrictive.\\
\subsubsection{Boundedness of the distribution function for fermions}
A distribution function of a fermion  has the additional upper bound $f<1$. Our scheme preserves this property which is shown in the following propositions.

\begin{proposition} \label{prop:fermion-smaller-1}
If $f_k$ represents the distribution function of a fermion with $f_k^\ell < 1$, the time discretization in Section \ref{subsec:firstordersplit} together with the space discretization described in Section \ref{sec:space} leads to $f_k^{\ell+1}<1$.
\end{proposition}
\begin{proof}
Let $f_k^\ell < 1$. The local equilibrium of a fermion is a Fermi-Dirac distribution function $\F$ for which $0<\F<1$ by definition. Hence, for the relaxation step it holds 
\begin{align} \label{eq:relax-smaller-1}
f_k^{*} = d_k f_k^\ell + d_k \Delta t\,(\tilde{\nu}_{kk}\F_{kk}^{*} + \tilde{\nu}_{kj} \F_{kj}^{*}) < d_k + d_k \Delta t (\tilde{\nu}_{kk} + \tilde{\nu}_{kj}) = 1.
\end{align}
Here, we used the definition of $d_k$ given by \eqref{di}. For the transport step \eqref{eq:transport_x}, the first-order fluxes lead with \eqref{eq:relax-smaller-1} to
\begin{align*}
    f_{k,i}^{\ell+1} &= ( 1-\frac{\Delta t}{m_k\Delta x}|p^1|) f_{k,i}^{\ast} + \frac{\Delta t}{m_k\Delta x}|p^1| f_{k,i-\operatorname{sign}(p^1)}^{\ast}< ( 1-\frac{\Delta t}{m_k\Delta x}|p^1|) + \frac{\Delta t}{m_k\Delta x}|p^1| =1
\end{align*}
For the second-order fluxes, define $\sigma:=\operatorname{sign}(f_{k,i}^\ast-f_{k,i-1}^\ast)$. We conclude that 
\begin{align*}
    \phi_{i+\frac{1}{2}}(f_k^\ast) \leq 
    \begin{cases}
        0 \quad &\text{if} \quad \sigma=-1 \\
        f_{k,i+1}^\ast-f_{k,i}^\ast \quad &\text{if} \quad \sigma=+1
    \end{cases},\\
    -\phi_{i-\frac{1}{2}}(f_k^\ast) \leq 
    \begin{cases}
        f_{k,i-1}^\ast-f_{k,i}^\ast \quad &\text{if} \quad \sigma=-1\\
        0 \quad &\text{if} \quad \sigma=+1 
    \end{cases}.
\end{align*}
With \eqref{eq:transport_x}, it follows that
\begin{equation*}
\begin{aligned}
   & f_{k,i}^{\ell+1} 
    = ( 1-\frac{\Delta t}{m_k\Delta x}|p^1|) f_{k,i}^{\ast} + \frac{\Delta t}{m_k\Delta x}|p^1| f_{k,i-\operatorname{sign}(p^1)}^{\ast} + \frac{\Delta t}{m_k\Delta x} \frac{|p^1|}{2} (\phi_{i+\frac{1}{2}}(f_k^\ast)-\phi_{i-\frac{1}{2}}(f_k^\ast)) \\
    &\leq ( 1-\frac{\Delta t}{m_k\Delta x}|p^1|) f_{k,i}^{\ast} + \frac{\Delta t}{m_k\Delta x}|p^1| f_{k,i-\operatorname{sign}(p^1)}^{\ast} + \frac{\Delta t}{m_k\Delta x} \frac{|p^1|}{2}
    \begin{cases}
        (f_{k,i-1}^\ast-f_{k,i}^\ast) ~ \text{if} ~ \sigma=-1 \\
        (f_{k,i+1}^\ast-f_{k,i}^\ast) ~ \text{if} ~ \sigma=+1
    \end{cases} \\
    &= ( 1-\frac{3}{2}\frac{\Delta t}{m_k\Delta x}|p^1|) f_{k,i}^{\ast} + \frac{\Delta t}{m_k\Delta x}|p^1| f_{k,i-\operatorname{sign}(p^1)}^{\ast} + \frac{\Delta t}{m_k\Delta x} \frac{|p^1|}{2}
    \begin{cases}
        f_{k,i-1}^\ast \; &\text{if} \quad \sigma=-1 \\
        f_{k,i+1}^\ast \; &\text{if} \quad \sigma=+1
    \end{cases} \\
    &\overset{\eqref{eq:relax-smaller-1}}< 1.
\end{aligned}
\end{equation*}
\end{proof}

\begin{proposition} \label{prop:fermion-smaller-1-ARS}
If $f_k$ represents the distribution function of a fermion with $f_k^\ell < 1$, the time discretization in Section \ref{subsec:secondorderIMEX} leads to $f_k^{\ell+1}<1$ for the space homogeneous case.
\end{proposition}
\begin{proof}
Let $f_k^\ell < 1$. The local equilibrium of a fermion is a Fermi-Dirac distribution function $\F$ for which $0<\F<1$ by definition. Hence,
{\small
\begin{align*}
    \begin{split}
        &f_k^{\ell+1} =  d_k ( f_k^\ell + \Delta t (1-\gamma) (\tilde{\nu}_{kk}\F_{kk}^{(1)}+\tilde{\nu}_{kj}\F_{kj}^{(1)}-(\tilde{\nu}_{kk}+\tilde{\nu}_{kj})f_k^{(1)}) ) + \gamma \Delta t d_k( \tilde{\nu}_{kk} \F_{kk}^{(2)} + \tilde{\nu}_{kj} \F_{kj}^{(2)} )\\
        &= d_k ( f_k^\ell (1-2\Delta t (1-\gamma)d_k) + \Delta t (1-\gamma)d_k (\tilde{\nu}_{kk}\F_{kk}^{(1)}+\tilde{\nu}_{kj}\F_{kj}^{(1)}) )  + \gamma \Delta t d_k( \tilde{\nu}_{kk} \F_{kk}^{(2)} + \tilde{\nu}_{kj} \F_{kj}^{(2)} )\\
        &< d_k \left[ 1-2\Delta t (1-\gamma)d_k + 2\Delta t (1-\gamma)d_k +  \gamma \Delta t (\tilde{\nu}_{kk}+\tilde{\nu}_{kj})\right] = 1.
    \end{split}
\end{align*}}
\end{proof}

\subsubsection{Conservation of mass, total momentum and total energy} \label{subsec:discreteconservation}

In this section, we concern the conservation of mass, total momentum, and total energy for the semi-discrete scheme. The proofs of the following propositions work analogously as and can be found in the proofs of Proposition 5.3 and 5.4 in \cite{velocity-dependent}.

\begin{proposition} \label{theo:conservation_relaxation}
The relaxation step in the first-order splitting scheme presented in Section \ref{subsec:firstordersplit} satisfies the conservation laws
\begin{gather*}
\int m_1 f_1^{\ast} dp = \int m_1 f_1^{\ell} dp , \quad \int  m_2 f_2^{\ast} dp = \int  m_2 f_2^{\ell} dp, \\
\int  \left( m_1 p f_1^{\ast} +  m_2 p f_2^{\ast} \right) dp
= \int \left(  m_1 p f_1^{\ell} + m_2 p f_2^{\ell}   \right) dp,\\
\int  \left(  \frac{|p|^2}{2m_1} f_1^{\ast} +  \frac{|p|^2}{2m_2}  f_2^{\ast} \right) dp
= \int  \left(  \frac{|p|^2}{2m_1}  f_1^{\ell} + \frac{|p|^2}{2m_2}  f_2^{\ell}   \right) dp.
\end{gather*}
\end{proposition}

\begin{proposition} \label{theo:conservation_transport}
For each $i=1,2$, the transport step in the first-order splitting scheme in Section \ref{subsec:firstordersplit}, combined with the space discretization presented in Section \ref{sec:space} satisfies the conservation laws
\begin{align*}
    \sum_{i=0}^I \int \mbp_k f^{\ell+1}_{k,i} dp \Delta x = \sum_{i=0}^I \int \mbp_k f^{\ast}_{k,i} dp \Delta x
\end{align*}
for periodic or zero boundary conditions. 
\end{proposition}

Since the second-order time-stepping scheme in Section \ref{subsec:secondorderIMEX} can be broken into relaxation and transport parts, each of which preserves the conservation of mass, total momentum, and total energy, we can state the following:
\begin{cor}
For periodic or zero boundary conditions, any combination of temporal and space discretization presented to Sections \ref{sec:time} and \ref{sec:space}, respectively, conserves mass, total momentum and total energy. 
\end{cor}

\subsubsection{Entropy inequality}

We study the entropy behavior for the first-order scheme in Section \ref{subsec:firstordersplit}. Both the relaxation and the transport step dissipate entropy; see Propositions \ref{theo:entropy_inequality_relaxation} and \ref{theo:entropy_inequality_transport}. Moreover, the minimal entropy is reached for the relaxation step if the distribution functions coincide with the corresponding target functions; see Proposition \ref{theo:entropy_equilibrium}. 

\begin{proposition}\label{theo:entropy_inequality_relaxation}
 Let $h_{\tau}$ be given by \eqref{entropy_h}. 
 The relaxation step in the first-order splitting scheme in Section \ref{subsec:firstordersplit} fulfills the discrete entropy inequality
 \begin{align*}
     \int  h_{\tau}(f_1^{\ast}) + h_{\tau'}(f_2^{\ast})dp \leq  \int h_{\tau}(f_1^{\ell}) + h_{\tau'}(f_2^{\ell}) dp.
 \end{align*}
\end{proposition}
\begin{proof}
By convexity
\begin{equation*}
h_{\tau_k}(f_k^{\ell}) \geq h_{\tau_k}(f_k^{\ast}) + h_{\tau_k}'(f_k^{\ast})(f_k^{\ell}- f_k^{\ast}).
\end{equation*}
For $f\geq 0$ ($\tau \in \{-1,0\})$, respective $0\leq f < 1$ ($\tau=+1$), the derivative
\begin{align*}
    h_{\tau}'(f) = \log \frac{f}{1-\tau f}
\end{align*}
is monotonically increasing such that
\begin{align} \label{eq:inequality}
(h_{\tau}'(x)-h_{\tau}'(y))(y-x)\leq 0 
\end{align}
for all  $x,y\geq 0$ ($\tau \in \{-1,0\})$ and  $0\leq x,y < 1$ ($\tau=+1$), respectively. Moreover, since
\begin{align*}
    h_{\tau_k}'(\E_{kk}^{\ast}) =  \alpha_k \cdot \mbp_k,
\end{align*}
it holds
\begin{align} \label{eq:zero-intra}
\int  h_{\tau_k}'(\E_{kk}^{\ast})\tilde{\nu}_{kk} (\E_{kk}^{\ast} -f_k^{\ast}) dp 
    =\int \alpha_k \cdot  \mbp_k \,\tilde{\nu}_{kk} (\E_{kk}^{\ast} -f_k^{\ast}) dp
    =0 
\end{align}
which vanishes as the conservation properties are satisfied at the semi-discrete level as well by
construction of the scheme.
Analogously for the inter-species terms,
\begin{equation} \label{eq:zero-inter}
\begin{aligned}
    \int &h_{\tau}'(\E_{12}^{\ast})\tilde{\nu}_{12}(\E_{12}^{\ast} -f_1^{\ast}) dp + \int  h_{\tau'}'(\E_{21}^{\ast})\tilde{\nu}_{21} (\E_{21}^{\ast} -f_2^{\ast}) dp\\
    &=\alpha_{12}^0 \int  \tilde{\nu}_{12} (\E_{12}^{\ast} -f_1^{\ast}) dp + \alpha_{21}^0 \int  \tilde{\nu}_{21}  (\E_{21}^{\ast} -f_2^{\ast}) dp \\ &\quad + \begin{pmatrix} \alpha^1 \\ \alpha^2 \end{pmatrix} \cdot \int \textcolor{black}{\left[  \tilde{\nu}_{12} (\E_{12}^{\ast} -f_1^{\ast})\begin{pmatrix} p \\ \frac{|p|^2}{2m_1} \end{pmatrix} +   \tilde{\nu}_{21}(\E_{21}^{\ast} -f_2^{\ast})  \begin{pmatrix} p \\ \frac{|p|^2}{2m_2} \end{pmatrix} \right]}dp  \\
    &=0.
\end{aligned}
\end{equation}
The implicit step \eqref{eq:update_split} is 
\begin{equation}
\label{eq:implicit_step_restated}
f_k^{\ast} - f_k^{\ell} = \Delta t\, \tilde{\nu}_{kk}(\E_{kk}^{\ast} -f_k^{\ast}) + \Delta t\, \tilde{\nu}_{kj} (\E_{kj}^{\ast} -f_k^{\ast}).
\end{equation}
 Using \eqref{eq:implicit_step_restated} and the convexity of $h_{\tau}$ leads to
\begin{equation} \label{eq:h_convex}
\begin{aligned}
h_{\tau_k}(f_k^{\ast}) - h(f_k^{\ell}) &\leq \textcolor{black}{h'_{\tau_k}}(f_k^{\ast})(f_k^{\ast}- f_k^{\ell}) \\
    &\overset{\eqref{eq:implicit_step_restated}}{=}  \Delta t \, h_{\tau_k}'(f_k^{\ast})\tilde{\nu}_{kk}(\E_{kk}^{\ast} -f_k^{\ast}) + \Delta t \, \textcolor{black}{h_{\tau_k}'}(f_k^{\ast})\tilde{\nu}_{kj}(\E_{kj}^{\ast} -f_k^{\ast}).
\end{aligned}
\end{equation}
Thus after integrating \eqref{eq:h_convex} with respect to $p$ and making use of \eqref{eq:zero-intra} and \eqref{eq:zero-inter}, we obtain
\begin{equation*} 
\begin{aligned}
&\int  h_{\tau}(f_1^{\ast}) dp - \int h_{\tau}(f_1^{\ell}) dp + \int h_{\tau'}(f_2^{\ast}) dp - \int h_{\tau'}(f_2^{\ell}) dp \\
    &\leq \Delta t  \int (h_{\tau}'(f_1^{\ast})-h_{\tau}'(\E_{11}^{\ast}))\tilde{\nu}_{11}(\E_{11}^{\ast} -f_1^{\ast}) dp + \Delta t  \int (h_{\tau'}'(f_2^{\ast})-h_{\tau'}'(\E_{22}^{\ast}))\tilde{\nu}_{22}(\E_{22}^{\ast} -f_2^{\ast}) dp \\
    &\, + \Delta t  \int (h_{\tau}'(f_1^{\ast})-h_{\tau}'(\E_{12}^{\ast}))\tilde{\nu}_{12}(\E_{12}^{\ast} -f_1^{\ast}) dp + \Delta t  \int (h_{\tau'}'(f_2^{\ast})-h_{\tau'}'(\E_{21}^{\ast}))\tilde{\nu}_{21}(\E_{21}^{\ast} -f_2^{\ast}) dp \\
    &\leq 0.
\end{aligned}
\end{equation*}
The last inequality comes by \eqref{eq:inequality}.  
\end{proof}

\begin{proposition} \label{theo:entropy_equilibrium}
 The inequality in Proposition \ref{theo:entropy_inequality_relaxation} is an equality if and only if $f_1^{\ell} = \E_{12}^{\ell}$ and  $f_2^{\ell} = \E_{21}^{\ell}$.  In such cases $f_1^{\ast} = \E_{12}^{\ast}$ and  $f_2^{\ast} = \E_{21}^{\ast}$.
\end{proposition} 
\begin{proof}
The proof works analogously as and can be found in \cite{velocity-dependent}. 
\end{proof}

\begin{proposition}\label{theo:entropy_inequality_transport}
 Let $h_{\tau}$ \textcolor{black}{be} given by \eqref{entropy_h} .
 The transport step in the first-order splitting scheme in Section \ref{subsec:firstordersplit} combined with the first-order spatial discretization in Section \ref{sec:space} fulfills the discrete entropy inequality
\begin{align*}
    \sum_{i=0}^I \left\{ \int  h_{\tau}(f_{1,i}^{\ell+1}) + h_{\tau'}(f_{2,i}^{\ell+1})dp \right\} \Delta x
        \leq \sum_{i=0}^I  \left\{   \int \textcolor{black}{h_{\tau}}(f_{1,i}^{\ast}) + h_{\tau'}(f_{2,i}^{\ast}) dp \right\} \Delta x ,
 \end{align*}
 for periodic or zero boundary conditions, provided that 
  \begin{align}\label{CFL2}
     \Delta t \leq \frac{m_k\Delta x}{\max |p^1|}.
 \end{align}
\end{proposition}
\begin{proof}
We can apply the same proof as in \cite{velocity-dependent} because $h_{\tau}$ is convex.
\end{proof}

\noindent We combine the two propositions above and obtain the following: 
\begin{cor}
  For periodic or zero boundary conditions, the first-order splitting scheme \ref{subsec:firstordersplit} combined with the first-order numerical fluxes in Section \ref{sec:space} fulfills the discrete entropy inequality
 \begin{align*}
    \sum_{i=0}^I \left\{ \int  h_{\tau}(f_{1,i}^{\ell+1}) + h_{\tau'}(f_{2,i}^{\ell+1})dp \right\} \Delta x
        \leq \sum_{i=0}^I  \left\{   \int \textcolor{black}{h_{\tau}}(f_{1,i}^{\ast}) + \textcolor{black}{h_{\tau'}}(f_{2,i}^{\ast}) dp \right\} \Delta x
 \end{align*}
provided that \textcolor{black}{\eqref{CFL2}.}
\end{cor}

\subsection{Momentum discretization} \label{sec:velocity}
 Eventually, we discretize the momentum variable. We center the discrete momenta
  $p_q = (p_{q_1}^1, p_{q_2}^2,p_{q_3}^{3})^\top$, with $q=(q_1,q_2,q_3) \in \mathbb{N}^3_0$, around $ u_{\rm mix}$ with the mixture mean velocity 
\begin{align*}
    u_{\rm mix} = \frac{p_1+ p_2}{N_1+N_2 },
\end{align*}
and restrict them to a finite cube. This means, for each component $r \in \{1,2,3\}$,
\begin{align*}
p^{r} \in [m_k u_{{\rm mix}}^{r} - 6 m_k v_{{\rm{th}},k} ,\, m_k u_{{\rm mix}}^{r} + 6m_k v_{{\rm{th}},k} ],
\end{align*}
where $v_{{\rm{th}},k} = \sqrt{\frac{T_{\rm mix}}{m_k}}$ is the thermal velocity of species $i$ and
\begin{align} \label{eq:T_mix}
     T_{\rm mix} = \frac{n_1 T_1 + n_2 T_2}{n_1+n_2} + \frac{1}{3}\frac{N_1 N_2}{N_1+N_2}  \frac{|\frac{P_1}{N_1}-\frac{P_2}{N_2}|^2}{n_1 + n_2},
\end{align}
is the mixture temperature. An adequate resolution is ensured by the momentum mesh size $\Delta p_k = 0.25 m_k v_{{\rm{th}},k}$ in each direction, as in \cite{Mieussens2000}.  

We emphasize the advantage of the multi-species BGK model that it is possible to use different grids for each species/equation. This feature becomes beneficial when the species masses, and hence the thermal speeds, differ significantly.\\
\\
All momentum integrals are replaced by discrete sums using the trapezoidal rule, i.e.
\begin{align*}
   \int (\cdot) dp \approx \sum_q \omega_q (\cdot)_q (\Delta p_k)^3 ,
\end{align*}
where $\omega_q=\omega_{q_1}\omega_{q_2}\omega_{q_3}$ are the weights and
\begin{align*}
\omega_{q_p} = \begin{cases} 1 \quad \text{if } \min(q_p) < q_p < \max(q_p),\\
\frac{1}{2} \quad \text{else}.
\end{cases}
\end{align*}
We need to distinguish between discrete and continuous moments, especially when determining the discrete local equilibria $\E_{kk,q}$ and $\E_{kj,q}$.  Since the minimization of \eqref{eq:potentialfunction_single} and \eqref{eq:potentialfunction_mixed} is solved using a discrete momentum grid and discrete moments $\bar{\mu}_{k},\bar{\mu}$ as input, the parameters $\alpha_{kk}$ and $\alpha_{kj}$ are determined such that $\E_{kk,q}$ and $\E_{kj,q}$ have the desired discrete moments. Thus, the conservation and entropy properties are fulfilled at the discrete level. (A similar approach for the standard single-species BGK equation is given in \cite{Mieussens2000}.)

\begin{theorem} \label{theo:properties_v_discrete}
Propositions \ref{prop:positivity_1st_order}, \ref{prop:positivity_2nd_order}, and \ref{theo:entropy_inequality_relaxation}-\ref{theo:entropy_inequality_transport} all hold true after replacing continuous integrals by their respective quadratures. 
Additionally, the scheme in Section \ref{sec:generaltime} satisfies the following conservation properties for $\ell\geq 0$ 
{\small
\begin{align*}
    \sum_{i,q}  \omega_q\left(f_{1,iq}^\ell \mbp_{1,q} (\Delta p_1)^3+ f_{2,iq}^\ell \mbp_{2,q}(\Delta p_2)^3\right)  \Delta x = \sum_{i,q}  \omega_q \left(f_{1,iq}^0 \mbp_{1,q} (\Delta p_1)^3+ f_{2,iq}^0 \mbp_{2,q}(\Delta p_2)^3\right)  \Delta x,
\end{align*}}
with $\mbp_{k,q} = (1,p_q,\frac{|p_q|^2}{2m_k})^\top$ and $f_{k,iq}^\ell \approx f_{k,i}^\ell(p_q)$.
\end{theorem}

\paragraph{Optimization algorithm} The minimization of \eqref{eq:potentialfunction_single} and \eqref{eq:potentialfunction_mixed} is solved by Newton's method which requires the evaluation of the gradients
\begin{align*} 
    \nabla_{\alpha_k} \varphi_k &\approx -\sum_q \omega_q d_{k,q} \tilde{\nu}_{kk,q}\,\E_{kk,q}\, \mbp_{k,q} (\Delta p_k)^3 +  \bar{\mu}_{k}, \\
    \nabla_{\alpha} \varphi &\approx   -\sum_q \omega_q d_{1,q}\nu_{12,q} \, \E_{12,q} \,\mbp_{12,q} (\Delta p_1)^3 -\sum_q \omega_q d_{2,q} \tilde{\nu}_{21,q}\, \E_{21,q}\, \mbp_{21,q}(\Delta p_2)^3
    + \bar{\mu},   
 \end{align*}   
and the Hessians
 \begin{align*}
    \nabla_{\alpha_k}^2 \varphi_k &\approx   \sum_q \omega_q d_{k,q} \tilde{\nu}_{kk,q} \, \zeta(\E_{kk,q},\tau_k)\, \mbp_{k,q} \otimes \mbp_{k}(\Delta p_k)^3, \\
    \nabla_{\alpha}^2 \varphi &\approx  \sum_q \omega_q d_{1,q} \tilde{\nu}_{12,q} \, \zeta(\E_{12,q},\tau_1) \,\mbp_{12}\otimes \mbp_{12,q}(\Delta p_1)^3 \\&+  \sum_q \omega_q d_{2,q} \tilde{\nu}_{21,q}\, \zeta(\E_{21,q},\tau_2)\, \mbp_{21}\otimes \mbp_{21,q}(\Delta p_2)^3 ,
\end{align*}
where
$\mbp_{12,q} = (1,0,p_{1,q},\frac{|p_{1,q}|^2}{2m_1})^\top$, $\mbp_{21,q}=  (0,\mbp_{2,q})^\top$ and
\begin{align*}
    \zeta(g,\tau) = \begin{cases}
     g \quad &\text{for } \tau = 0, \\
     g^2 e^{-\alpha\cdot \mbp_i} \quad &\text{for } \tau = \pm 1.
    \end{cases}
\end{align*}
The input data in \eqref{eq:input_mu_i} is computed in a straight-forward way:
\begin{align*}
    \bar{\mu}_{k} \approx \sum_q \omega_q d_{k,q} \tilde{\nu}_{kk,q} \,G_{k,q}\, \mbp_{k,q} (\Delta p_k)^3.
\end{align*}
Analogously for the input data $\bar{\mu}$ in \eqref{eq:input_mu}.

\section{Numerical results}
In this section, we present several numerical tests. We illustrate the properties of our model and demonstrate the properties of our scheme.

\label{sec:results}
\subsection{Relaxation in a homogeneous setting}
\subsubsection{Decay rates and illustration of the schemes' properties} \label{test:decay-rates}
We validate our numerical scheme for quantum particles and verify the decay rates for the mean velocities and kinetic temperatures which are  given analytically in Section \ref{sec:macro-equations}.\\
\\
Initially, we set the distribution functions to Maxwellians 
\begin{align} \label{eq:Maxwellian}
    f_k = \mathcal{M}[n_k,U_k,T_k,m_k] = \frac{n_k}{(2\pi T_k m_k)^{3/2}} \exp \left( -\frac{|p - m_k U_k|^2}{2T_k m_k} \right) .
\end{align}
with 
\begin{align*}
    m_1 &= 1.0, \quad n_1 = 1.0, \quad U_1 = (0.5,0,0)^\top, \quad T_1 = 1.0, \\
    m_2 &= 1.5, \quad  n_2 = 1.2, \quad  U_2 = (0.1,0,0)^\top, \quad  T_2 = 0.5.
\end{align*}
These initial data are chosen to only illustrate the basic properties of the model and scheme, respectively, but we do not incorporate further physical details (e.g. for a specific quantum regime). The collision frequencies are set to $\tilde{\nu}_{kj}=1$.

For the simulation, we use a momentum grid with $48^3$ nodes and the first-order splitting scheme from Section \ref{subsec:firstordersplit} with the time step $\Delta t = 0.01$.\\
\\
We study any combination of classical particles, fermions and bosons. \textcolor{black}{For example, in} the interactions of fermions with fermions, we illustrate the evolution of the entropy and the entropy dissipation in Figure \ref{fig:entropy_quantum}. In Figure \ref{fig:conservation_quantum}, we demonstrate the conservation properties  where the numerical oscillations in mass, total momentum and total energy are only of the order $10^{-14}$.

In Figure \ref{fig:decay_u_quantum}, we verify the behavior of the mean velocities converging exponentially fast to a common value. The numerical decay rate and the analytical one \eqref{eq:quantum-decay-momentum} coincide very well. We only display the rate for the interactions of fermions with fermions because the decay rate is independent of the type of the species.

In Figure \ref{fig:decay_T_quantum}, we consider the behavior of the temperatures where we distinguish between the kinetic temperatures \textcolor{black}{$T_k$ and the physical temperatures $\vartheta_k$ of the fluid.
\begin{align} \label{eq:T_kinetic}
T_k = \frac{2}{3}\left(\frac{E_k}{ n_k}- \frac{1}{2} \frac{|P_k|^2}{m_k n_k^2}\right), \quad \vartheta_k = \frac{1}{2a_k},
\end{align} 
where $a_k$ is defined in \eqref{ai} and
$\vartheta_k$ corresponds to the temperature defined in \cite{HuJin2011}.} 
In the first column, we observe that the kinetic temperatures do not converge to a common value whenever a quantum particle is involved. 
This is also visible in the second column. The numerical and analytical decay rates for the kinetic temperatures coincide very well, and the difference converges to a constant value for quantum particles. Such behavior of the kinetic temperatures for quantum particles comes by an additional term for the decay rates \eqref{eq:quantum-decay-kinetic-T} which vanishes for classical-classical interactions, see Remark \ref{remark:classic-classic-2}. Additionally, we compare the results to the physical temperatures $\vartheta_k$. Even though the kinetic temperatures behave differently for quantum particles, the physical temperatures converge to a common value in all cases as predicted by the theory. 

{
\begin{figure}[htb]
\centering
\hspace{0.9cm} \begin{subfigure}[c]{0.35\textwidth}
\includegraphics[width=\textwidth]{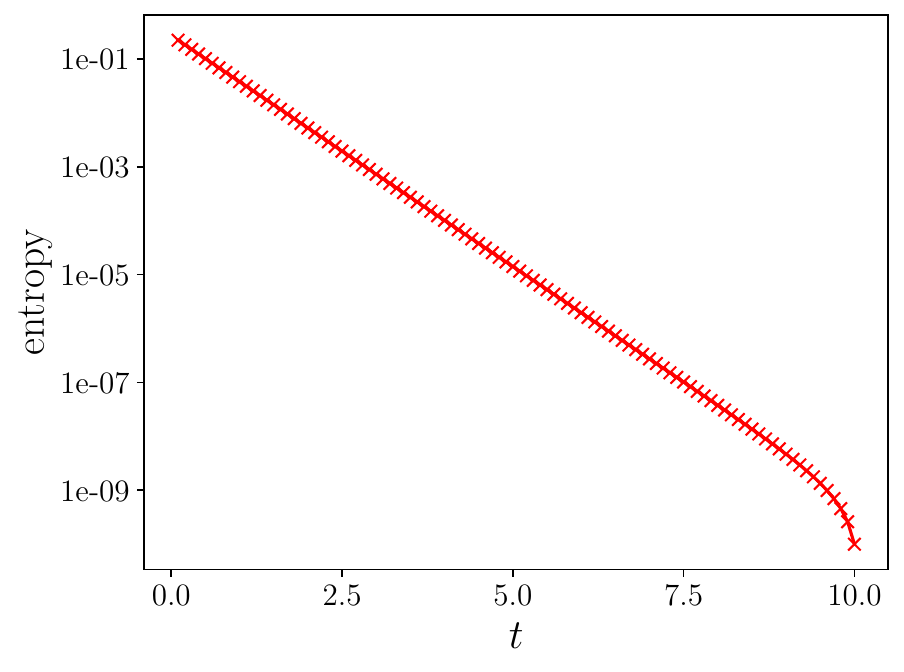}
\subcaption{entropy}
\end{subfigure}
\hspace{0.9cm}
\begin{subfigure}[c]{0.35\textwidth}
\includegraphics[width=\textwidth]{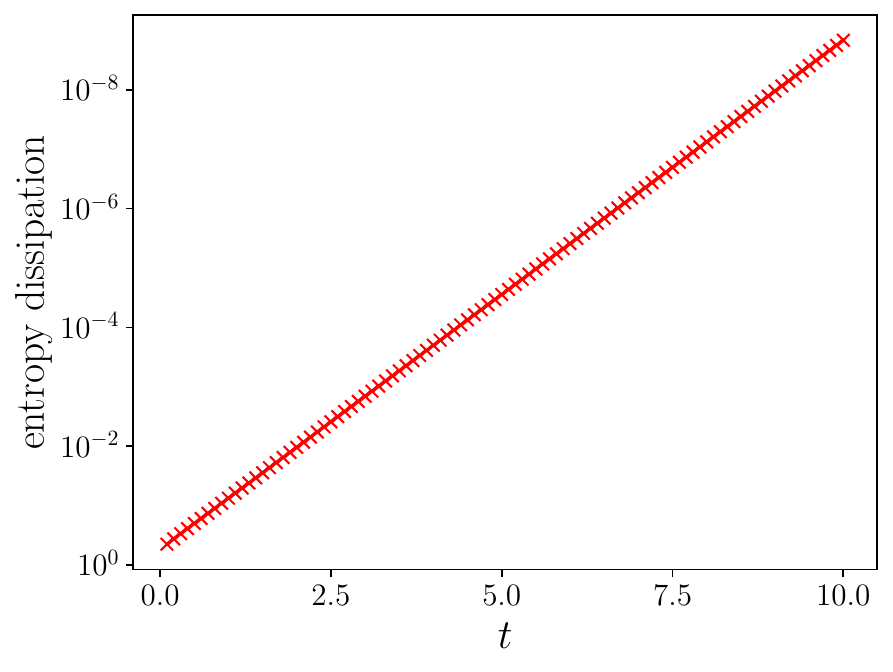}
\subcaption{entropy dissipation }
\end{subfigure}
    \caption{Entropy and entropy dissipation for the test case in Section \ref{test:decay-rates}, exemplary for fermion-fermion interactions. The entropy decays monotonically.
    }
    \label{fig:entropy_quantum}
\end{figure}}

{\scriptsize
\begin{figure}[htb]
\centering
\begin{subfigure}[c]{0.3\textwidth}
\includegraphics[width=\textwidth]{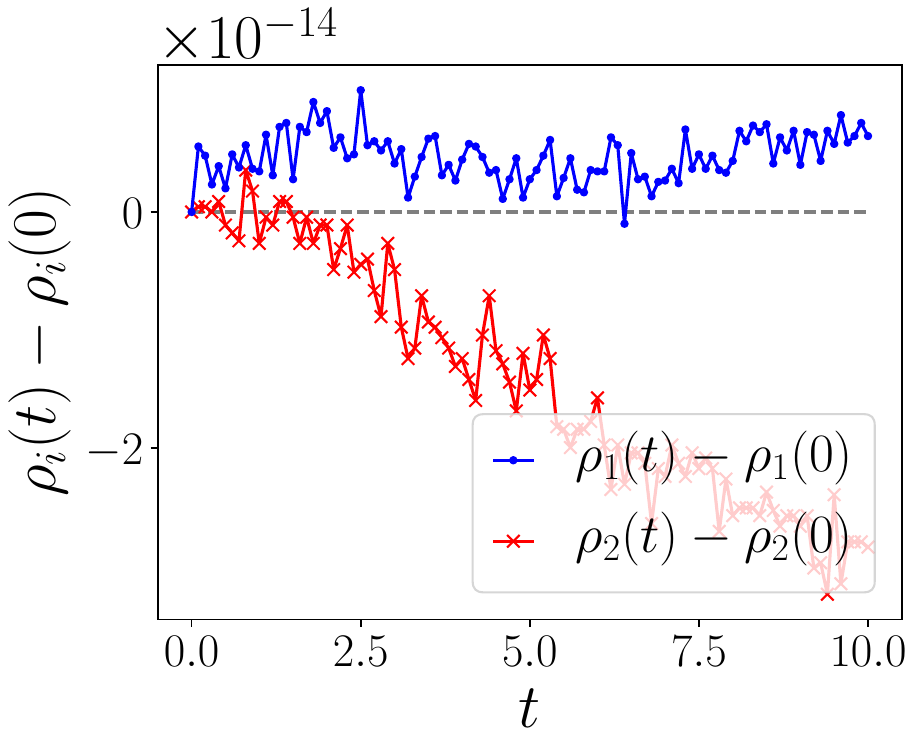}
\subcaption{species mass}
\end{subfigure}
\hfill
\begin{subfigure}[c]{0.3\textwidth}
\includegraphics[width=\textwidth]{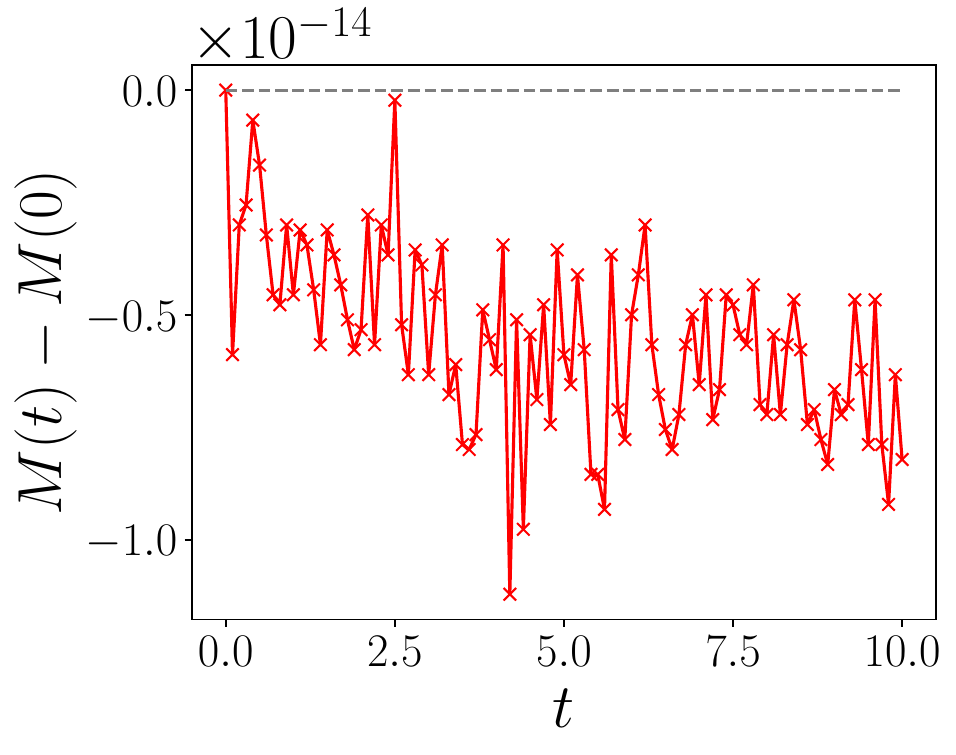}
\subcaption{total momentum}
\end{subfigure}
\hfill
\begin{subfigure}[c]{0.3\textwidth}
\includegraphics[width=\textwidth]{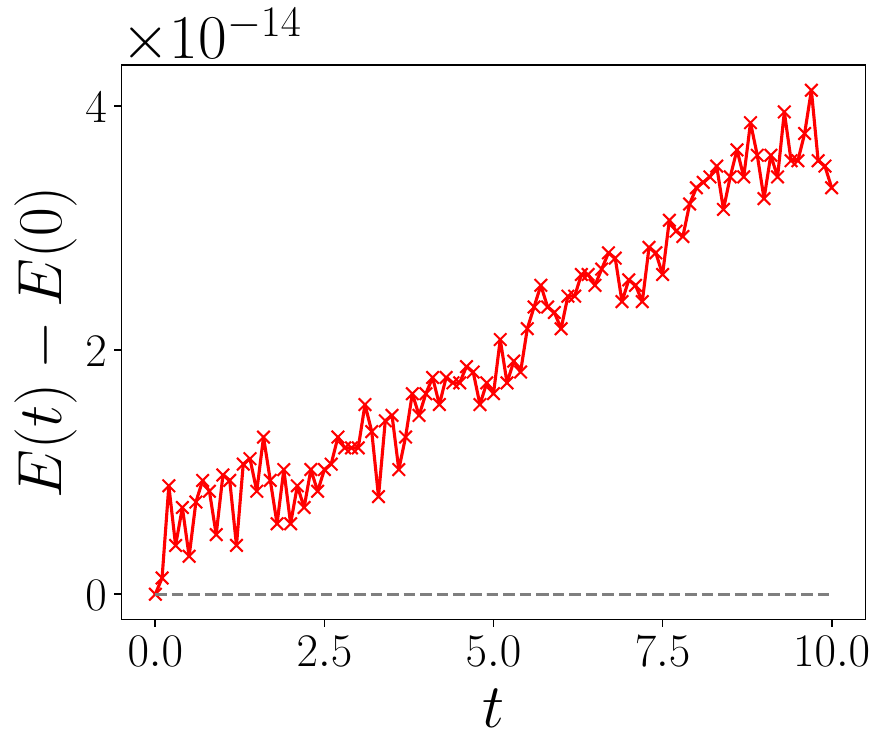}
\subcaption{total energy}
\end{subfigure}
    \caption{Illustration of the conservation properties for the test case in Section \ref{test:decay-rates}, exemplary for fermion-fermion interactions.  The mass densities of each species ($\rho_k=m_k n_k$), 
    the total momentum  ($M$) and total energy ($E$) have small oscillations of the order of $10^{-14}$.  }
    \label{fig:conservation_quantum}
\end{figure}}

\begin{figure}[htb] 
\centering
\hspace{0.6cm} \begin{subfigure}[c]{0.35\textwidth}
\includegraphics[width=\textwidth]{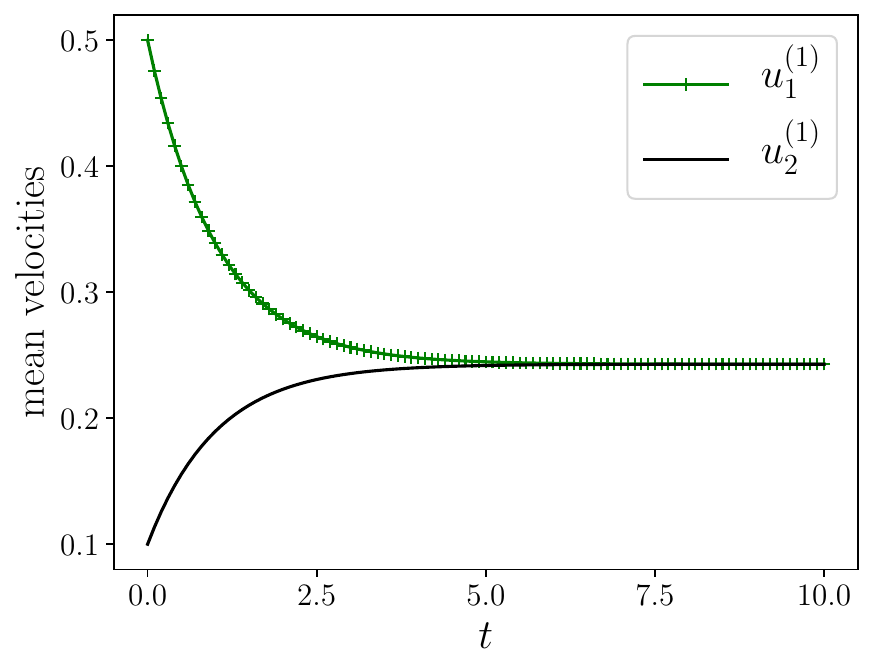}
\end{subfigure}
\hspace{0.6cm}
\begin{subfigure}[c]{0.35\textwidth}
\includegraphics[width=\textwidth]{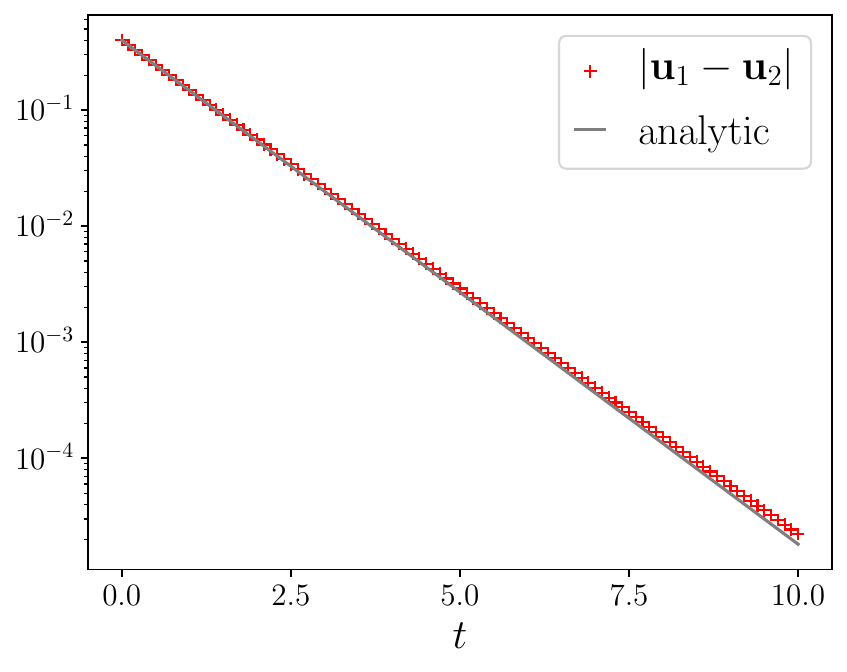}
\end{subfigure}
    \caption{Mean velocities for the test case in Section \ref{test:decay-rates}, exemplary for fermion-fermion interactions. The mean velocities converge exponentially fast to a common value, and the numerical decay rate coincides very well with the analytical one. 
    }
    \label{fig:decay_u_quantum}
\end{figure}

\begin{figure}[htb]
\centering
\begin{subfigure}[r]{0.20\textwidth}
    fermion-fermion:
\end{subfigure}
\hfill
\begin{subfigure}[c]{0.24\textwidth}
\includegraphics[width=\textwidth]{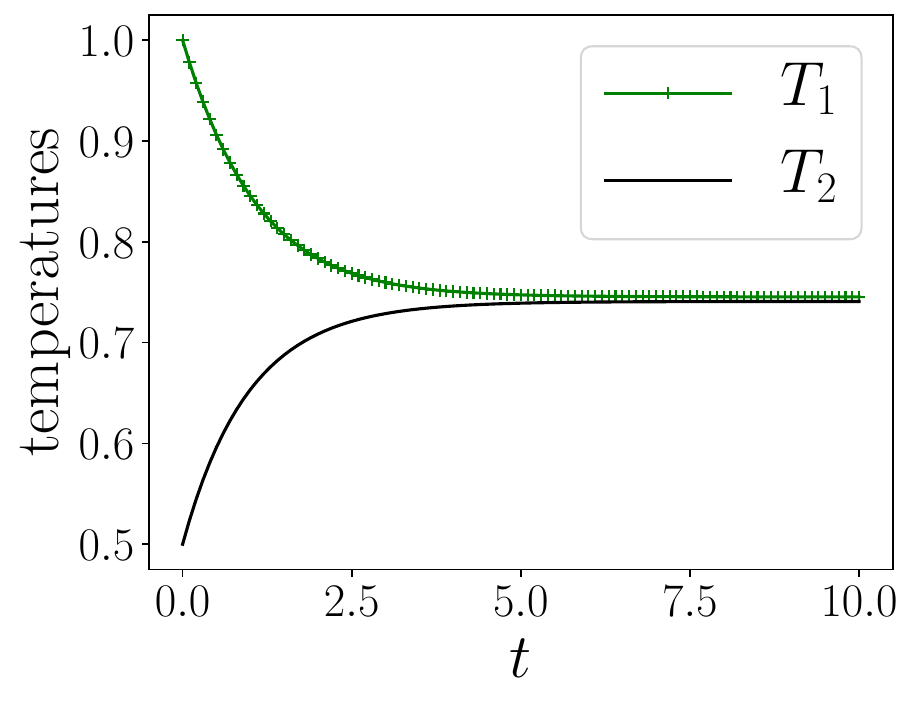}
\end{subfigure}
\hfill
\begin{subfigure}[c]{0.24\textwidth}
\includegraphics[width=\textwidth]{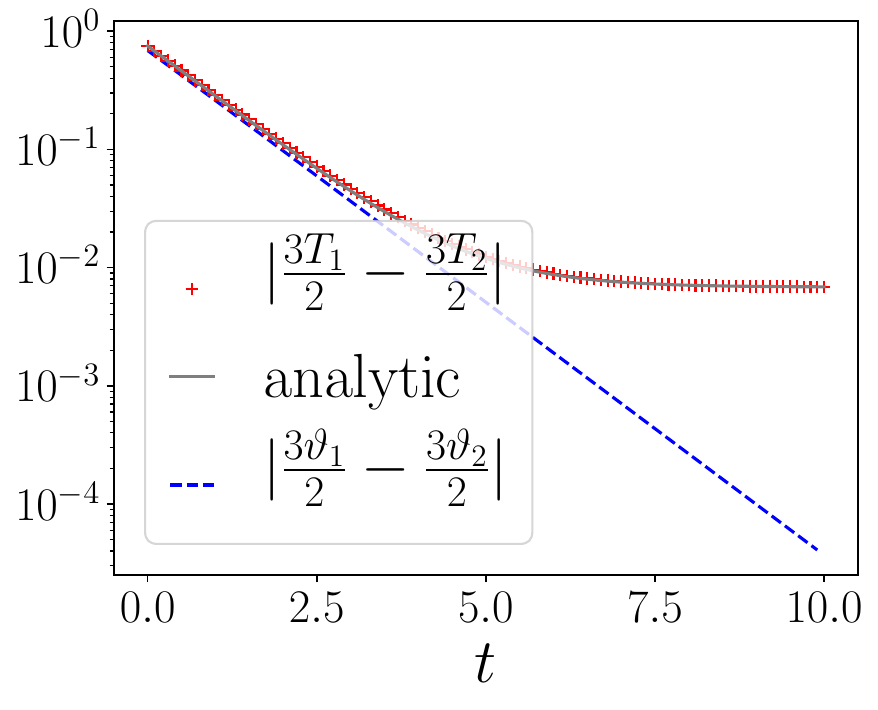}
\end{subfigure}

\begin{subfigure}[r]{0.2\textwidth}
    boson-fermion:
\end{subfigure}
\hfill
\begin{subfigure}[c]{0.24\textwidth}
\includegraphics[width=\textwidth]{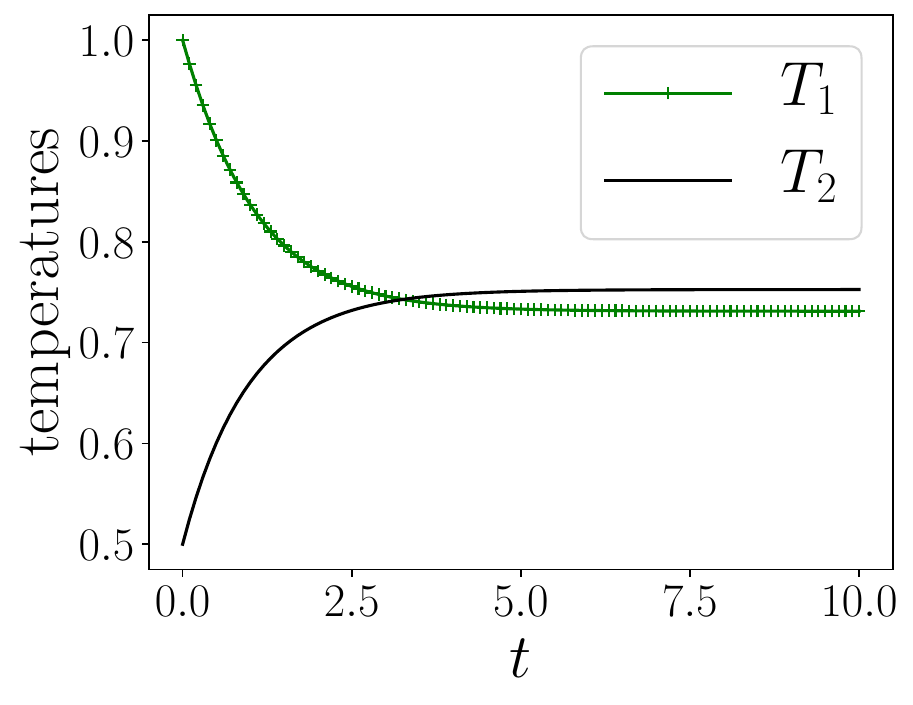}
\end{subfigure}
\hfill
\begin{subfigure}[c]{0.24\textwidth}
\includegraphics[width=\textwidth]{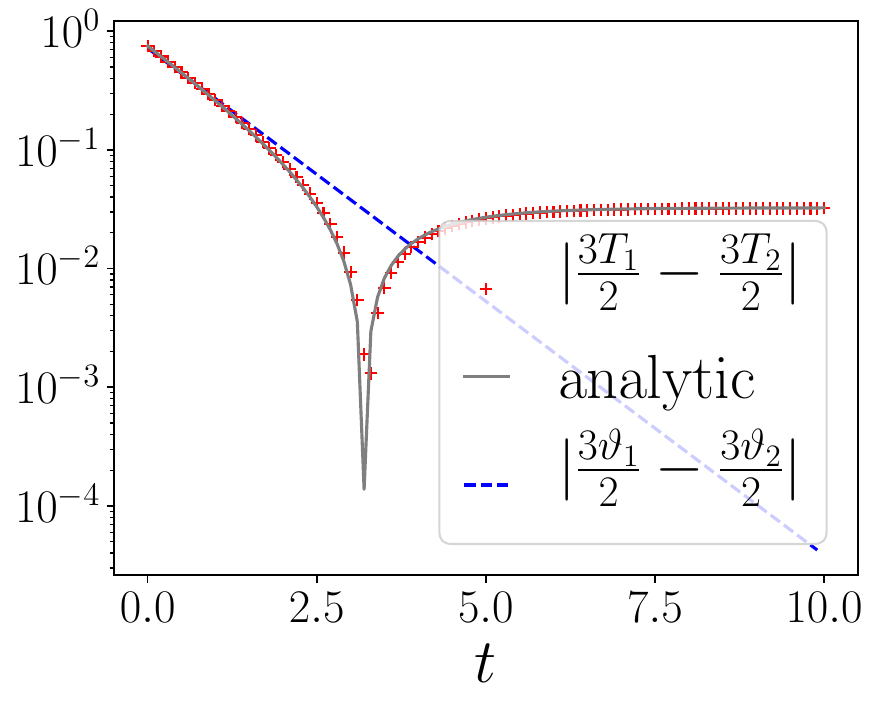}
\end{subfigure}

\begin{subfigure}[r]{0.2\textwidth}
    fermion-classic:
\end{subfigure}
\hfill
\begin{subfigure}[c]{0.24\textwidth}
\includegraphics[width=\textwidth]{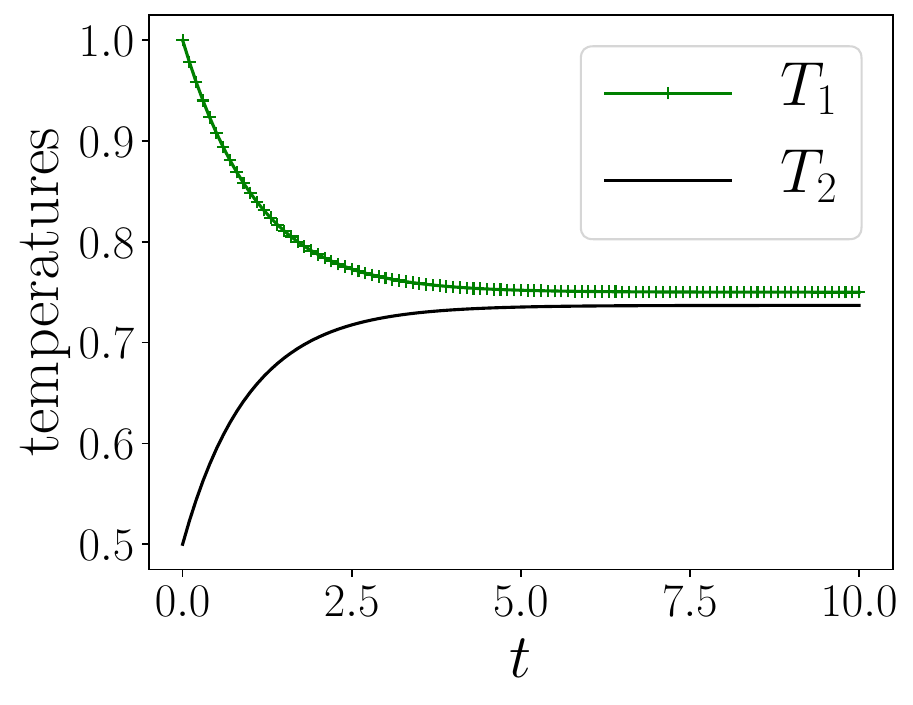}
\end{subfigure}
\hfill
\begin{subfigure}[c]{0.24\textwidth}
\includegraphics[width=\textwidth]{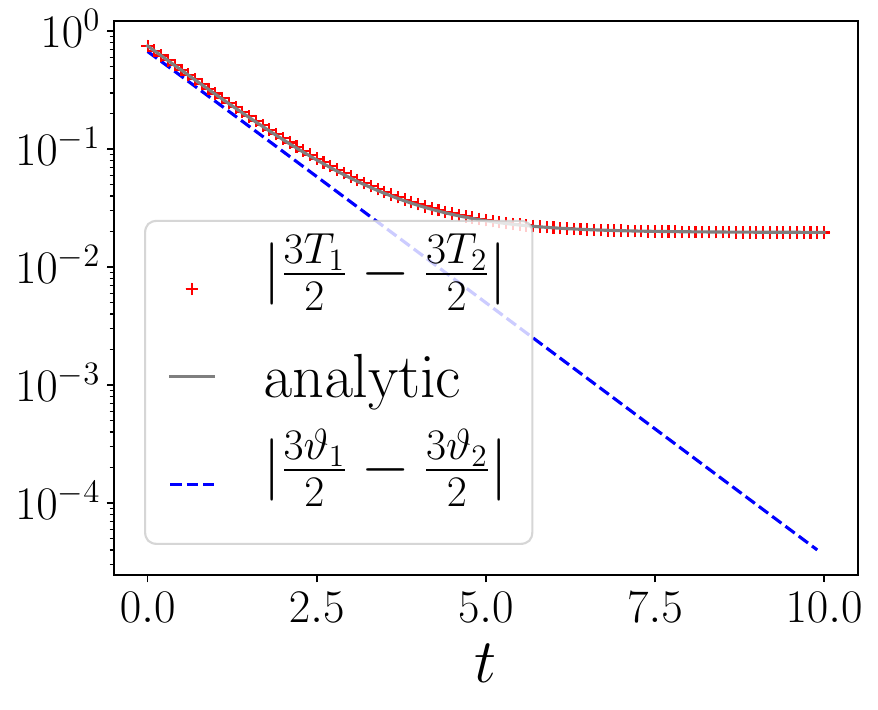}
\end{subfigure}

\begin{subfigure}[r]{0.2\textwidth}
    boson-boson:
\end{subfigure}
\hfill
\begin{subfigure}[c]{0.24\textwidth}
\includegraphics[width=\textwidth]{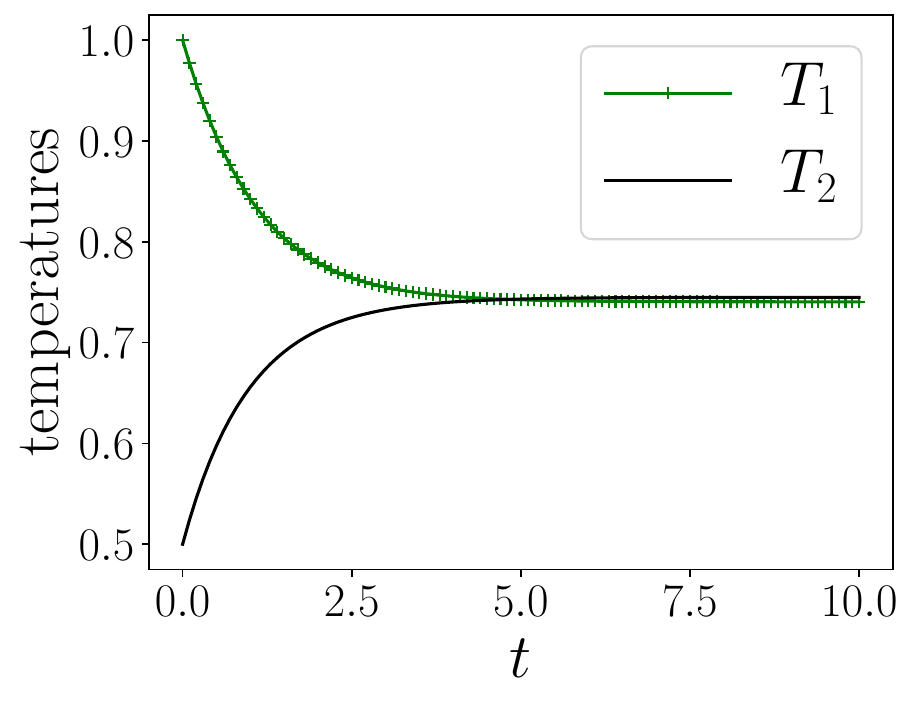}
\end{subfigure}
\hfill
\begin{subfigure}[c]{0.24\textwidth}
\includegraphics[width=\textwidth]{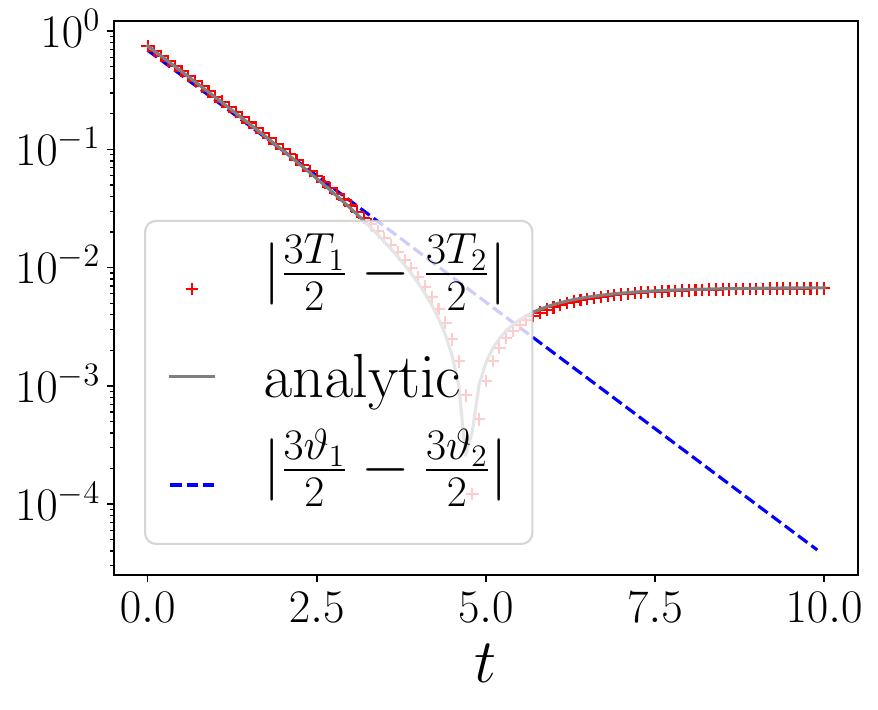}
\end{subfigure}

\begin{subfigure}[r]{0.2\textwidth}
    boson-classic:
\end{subfigure}
\hfill
\begin{subfigure}[c]{0.24\textwidth}
\includegraphics[width=\textwidth]{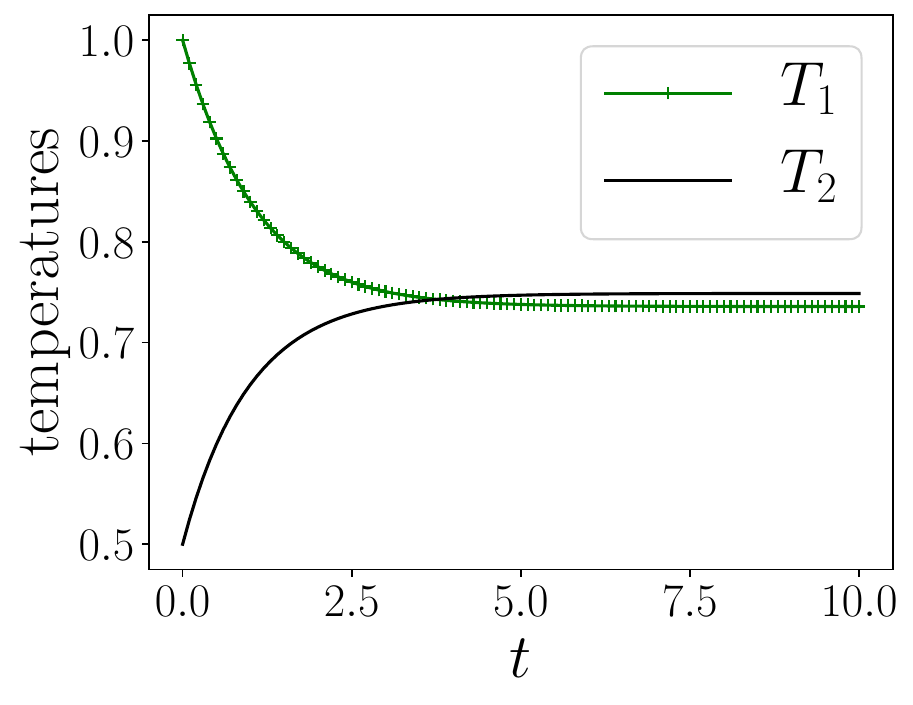}
\end{subfigure}
\hfill
\begin{subfigure}[c]{0.24\textwidth}
\includegraphics[width=\textwidth]{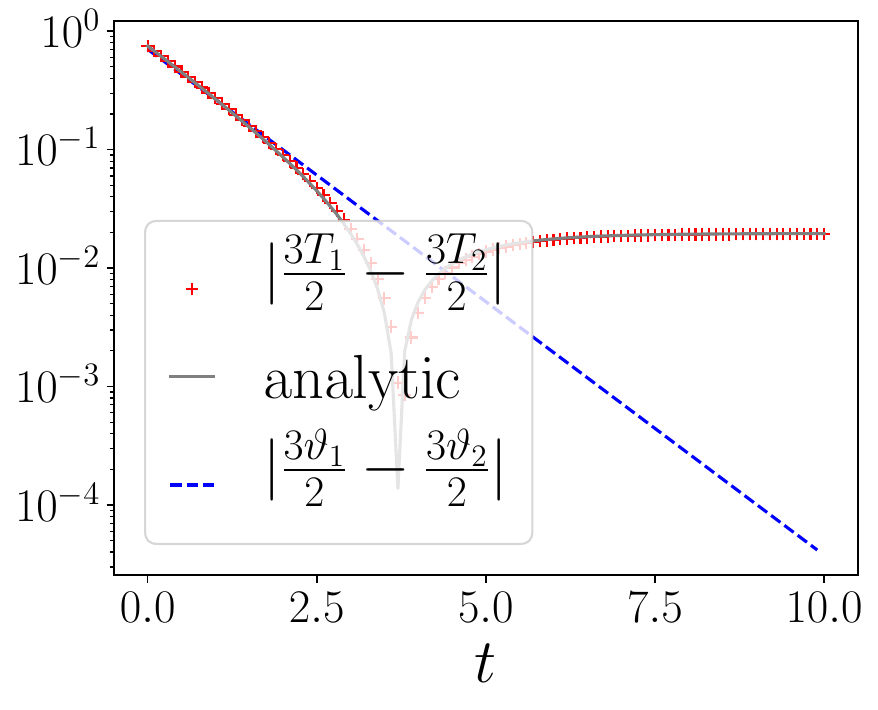}
\end{subfigure}

\begin{subfigure}[r]{0.2\textwidth}
    classic-classic:
\end{subfigure}
\hfill
\begin{subfigure}[c]{0.24\textwidth}
\includegraphics[width=\textwidth]{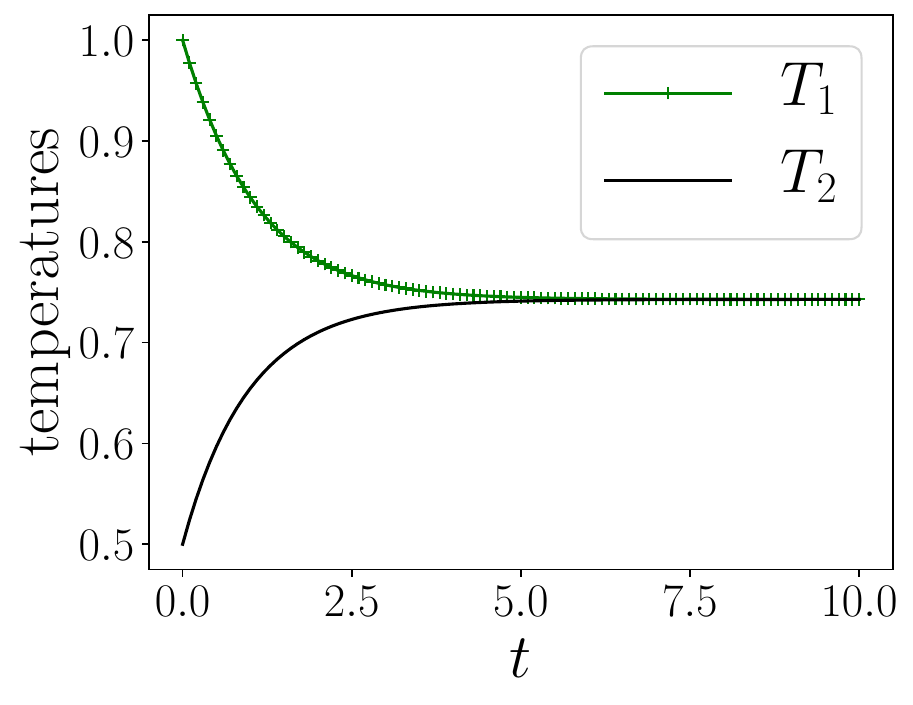}
\end{subfigure}
\hfill
\begin{subfigure}[c]{0.24\textwidth}
\includegraphics[width=\textwidth]{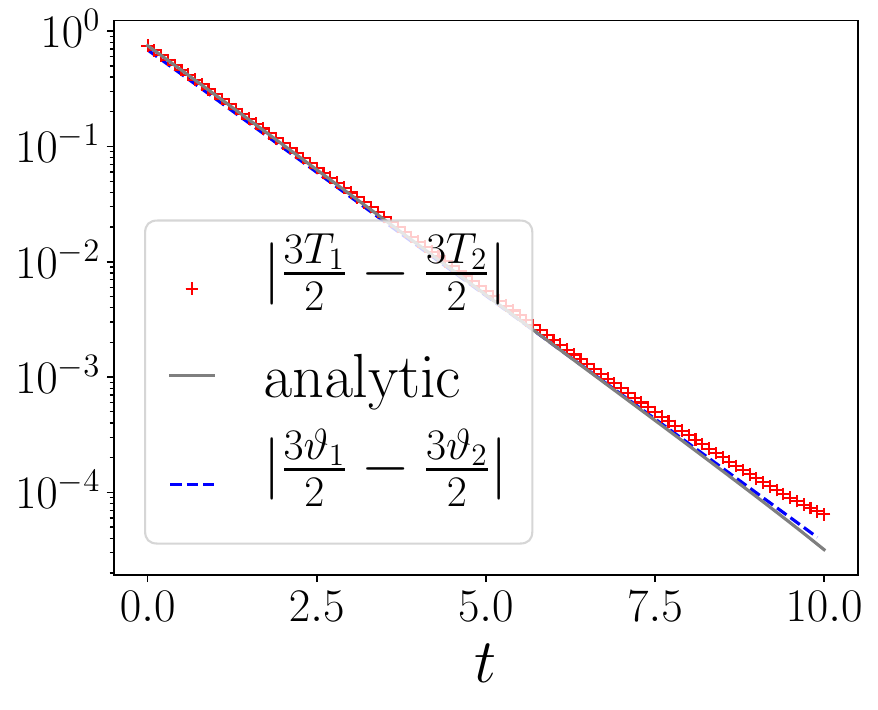}
\end{subfigure}

    \caption{Evolution of the temperatures for the test case in Section \ref{test:decay-rates}. First column: kinetic temperatures $T_k$; whenever a quantum particle is involved, the kinetic temperatures do not converge to a common value. Second column: decay rates for kinetic temperatures in logarithmic scale --- numerical and analytical values coincide very well. Additionally, the difference between the physical temperatures $\vartheta_k$ is displayed which decays exponentially fast, whereas the kinetic temperatures $T_k$ behave differently for quantum particles. }
    \label{fig:decay_T_quantum}
\end{figure}
\FloatBarrier

\subsubsection{Sulfur-Flourine-electrons test case}
\label{test:SFe-quantum}

We run a space homogeneous, 3-species test case inspired by \cite{HHM}. In the following, the index $S$ refers to sulfur ions, the index $F$ refers to fluorine ions, and the index $e$ refers to electrons. For convenience, we clarify in the Appendix how the model and the numerical scheme can be extended straight-forwardly to more than two species.

We incorporate collision frequencies $\nu_{kj}=0.00753 \frac{1}{{\rm fs}}$, \textcolor{black}{where $1{\rm fs} = 10^{-15} {\rm s}$,} which are approximately of the same order as those used in \cite{HHM}. The masses of the species are
\begin{align*}
m_S &= 32.07 \text{u} - 11 m_e,  \quad m_F = 19 \text{u} - 7 m_e, \quad m_e = 9.11\cdot 10^{-28} \text{g} 
\end{align*}
with the atomic mass  $\text{u} = 1.6605\cdot 10^{-24}\,\text{g}$. The ions are treated like classical particles and initialized by $f_k = \mathcal{M}[n_k,U_k,T_k,m_k]$ $(k=S,F)$ with
\begin{align*}
n_S &= 10^{19}\,\text{cm}^{-3}, \quad n_F = 6\cdot 10^{19}\,\text{cm}^{-3}, \\
U_S &= U_F = 0  \,\frac{\text{cm}}{\text{s}},\\
T_S &= T_F = 15\,\text{eV},
\end{align*}
where $M$ is defined in \eqref{eq:Maxwellian}.
For the electrons, we compare the behavior when they are treated like classical particles to the behavior when they are treated like fermions. In the former case, we initialize $f_e = \mathcal{M}[n_e,U_e,\vartheta_e,m_e]$ with
\begin{align*}
n_e = 53\cdot 10^{19}\,\text{cm}^{-3}, \quad
U_e = 0  \,\frac{\text{cm}}{\text{s}}, \quad
\vartheta_e = 100\,\text{eV}.
\end{align*}
It holds $T_e=\vartheta_e$ for classical particles. In the latter case --- electrons being treated as fermions --- we initialize the distribution function by a Fermi-Dirac function, but we keep the same macroscopic quantities, i.e.
\begin{align*}
    f_e = \left[ \frac{(2 \pi m_e \vartheta_e)^{3/2}}{\alpha\, n_e} e^{\frac{|p|^2}{2m_e\vartheta_e}} +1 \right]^{-1}
\end{align*}
with the scaling factor $\alpha=1.061711634$ which leads to the desired $\int f_e dp = n_e$.

We use  momentum grids with $48^3$ nodes for each species, and we use the second-order IMEX RK scheme from Section \ref{subsec:secondorderIMEX} with time step $\Delta t = 0.1$ fs. \\
\\
We illustrate the evolution of the temperatures in Figure \ref{fig:SFe}. For the purely classic test case, the physical and the kinetic temperatures coincide such that the temperature in equilibrium $T_{\rm eq}$ can be precomputed from the initial data \cite{velocity-dependent}:
\begin{align} \label{eq:T_eq}
    T_{\rm eq} = T_{\rm mix}(0) \overset{\eqref{eq:T_mix}}{=} \frac{n_1 T_1(0) + n_2 T_2(0) + n_3 T_3(0)}{n_1+n_2+n_3}. 
\end{align} 
In Figure \ref{fig:SFe}, we observe that all species temperatures converge to that value for the classical simulation. Additionally, we display the results when we consider the electrons to be fermions instead. As predicted by the theory, the physical temperatures converge to a common value. However, the physical temperatures generally differ from the kinetic temperatures in the quantum case. As a consequence, the physical temperature in equilibrium does not equal $T_{\rm eq}$.

\begin{figure}[htb]
\centering
\includegraphics[width=0.75\textwidth]{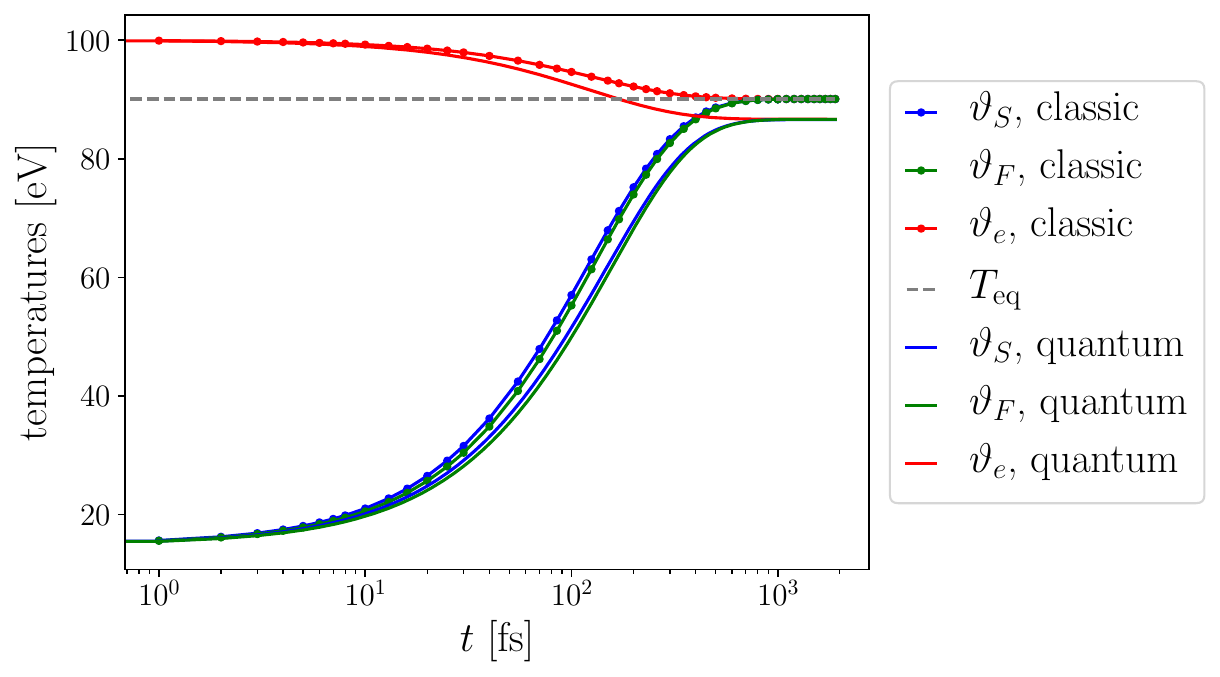}

    \caption{Evolution of the physical temperatures for the Sulfur-Fluorine-electrons quantum test case in Section \ref{test:SFe-quantum}. When both the ions and the electrons are treated classically (lines with dots), the physical temperatures (which coincide with the kinetic temperatures \eqref{eq:T_kinetic}) converge to the mixture temperature $T_{\rm eq}$ defined in \eqref{eq:T_eq}. When the electrons are treated like fermions instead, the physical temperatures do converge to a common value as predicted by the theory. However, this value differs from $T_{\rm eq}$.}
    \label{fig:SFe}
\end{figure}


\subsection{Sod problem} \label{test:Sod}
We run a quantum-kinetic version of the well-known Sod problem \cite{Sod} in the fluid regime for fermions. 
As carried out in \cite{FilbetHuJin}, the limiting equations for the kinetic equations in the fluid regime are the quantum Euler equations.

We implement a single-species test case with the multi-species model by assuming $m_1=m_2=m$, $n_1=n_2=n$, $U_1=U_2=U$ and $T_1=T_2=T$. We set $m=1$ and use $\tilde{\nu}_{kj}=2\cdot 10^4$ for approaching the fluid regime.
The initial data is given by $f_1 = f_2 = \mathcal{M}[n,u,T,m]$
where $\mathcal{M}$ is defined in \eqref{eq:Maxwellian} with 
\begin{align*}
n = 1, 
\qquad U = 0,
\qquad T = 1,
\end{align*}
for $x \leq 0$ and
\begin{align*}
n= 0.125, 
\qquad U = 0,
\qquad T = 0.8
\end{align*}
for $x>0$.

The simulations are run using a velocity grid with $48^3$ points and 300 equally spaced cells in $x$.  We use the second-order IMEX Runge-Kutta scheme from Section \ref{subsec:secondorderIMEX} combined with the second-order finite volume scheme from Section \ref{sec:space}.

Numerical results of the macroscopic quantities are given in Figure \ref{fig:Sod}. The fluid limit is recovered fairly well by the density $n$, mean velocity $u$ and kinetic temperature $T$. We see again that the physical temperature $\vartheta$ deviates from the kinetic temperature.

\begin{figure}[htb]
\centering
\begin{subfigure}[c]{0.45\textwidth}
\includegraphics[width=\textwidth]{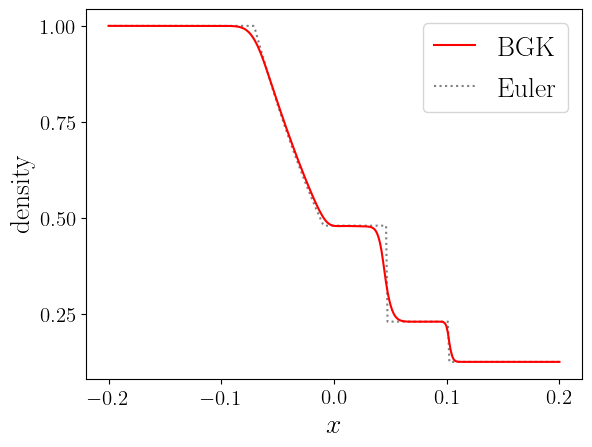}
\caption{density}
\end{subfigure}
~
\begin{subfigure}[c]{0.45\textwidth}
\includegraphics[width=\textwidth]{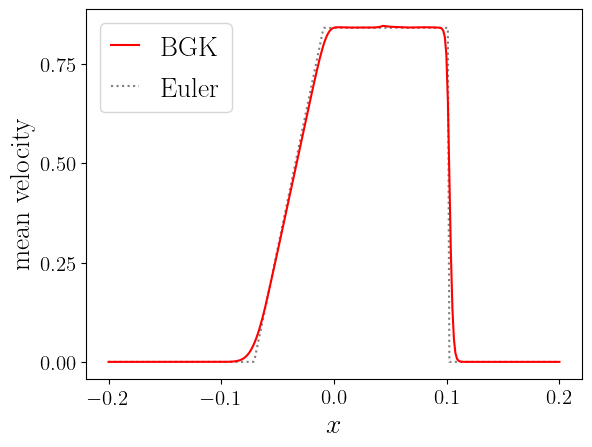}
\caption{mean velocity}
\end{subfigure}
\vskip10pt
\begin{subfigure}[c]{0.45\textwidth}
\includegraphics[width=\textwidth]{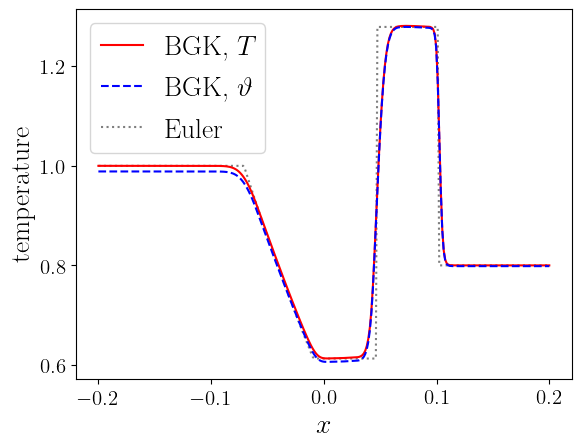}
\caption{temperature}
\end{subfigure}
~
    \caption{Numerical solution  at $t=0.055$ of the Sod problem in Section \ref{test:Sod}.   We show results for a 2-species kinetic simulation for fermions. The solutions for both species are identical; we show only the species 1 results.  For reference, the exact solution for the quantum Euler equations is also provided (dotted gray line).  The kinetic solution recovers the fluid limit fairly well. }
    \label{fig:Sod}
\end{figure}

\par{\bf Appendix.}\, 

\setcounter{section}{6}
\setcounter{subsection}{0}

\subsection{Extension to N-Species}
The two-species model can be extended to a system of $N$-species that undergo binary interactions. For ease in notation, we illustrate here the 3-species case. Each distribution function $f_k$, $k=1,\dots,3$, represents the solution to
\begin{align*}
    \partial_t f_k + \frac{p}{m_k}\cdot \nabla_x f_k = \tilde{\nu}_{k1}(\mathcal{K}_{k1}-f_k) + \tilde{\nu}_{k2}(\mathcal{K}_{k2}-f_k)+ \tilde{\nu}_{k3} (\mathcal{K}_{k3}-f_k)
\end{align*}
with $\tilde{\nu}_{kj} = \nu_{kj}n_j$.
Since we still consider only binary interactions, the properties in section \ref{sec:model}  are still satisfied.
\\

The presented numerical scheme is based on the general implicit solver in Section \ref{sec:generaltime}. Since the transport operators act only on the individual species, we focus on and shortly illustrate the scheme of the relaxation process.

As above, we write the implicit updates of the distribution functions in a generic steady state form
\begin{align} \label{eq:update_general_3}
f_k  = d_k G_k  + d_k \gamma\Delta t  (\tilde{\nu}_{kk} \E_{kk,\tau_{k}}  + \tilde{\nu}_{kj} \E_{kj,\tau_{j}} + \tilde{\nu}_{kl} \E_{kl,\tau_{l}})
\end{align}
for $k,j,l\in\{1,2,3\}$, each of $k,j,l$ distinct, where $\E_{kk,\tau_{k}}$, $\E_{kj,\tau_{j}}$ and $\E_{kl,\tau_{l}}$ are the unique attractors associated to $f_k$,
\begin{equation*}
    d_k = \frac{1}{1+\gamma\Delta t  (\tilde{\nu}_{kk} + \tilde{\nu}_{kj} +\tilde{\nu}_{kl})},
\end{equation*}
and $G_k$ is a known function.
When we can express $\E_{kk} $,  $\E_{kj} $ and $\E_{kl}$ as functions of $G_k$, $G_j$ and $G_l$, \eqref{eq:update_general_3} provides an explicit update formula for $f_k$.

We apply the conservation properties to \eqref{eq:update_general_3}. An analogous calculation as in the 2-species case leads to a set of constraints to determine  the attractors from the given data:
\begin{align} \label{eq:moments_A_3}
\begin{split}
&\int d_1  \left(\tilde{\nu}_{11} \E_{11,\tau_{1}}+\tilde{\nu}_{12} \E_{12,\tau_{1}} + \tilde{\nu}_{13} \E_{13,\tau_{1}} \right) \mbp_1  dp \\&+ \int d_2   \left(\tilde{\nu}_{21} \E_{21,\tau_{2}}+\tilde{\nu}_{22} \E_{22,\tau_{2}} + \tilde{\nu}_{23} \E_{23,\tau_{2}}\right) \mbp_2dp \\
+  &\int d_3   \left(\tilde{\nu}_{31} \E_{31,\tau_{3}}+\tilde{\nu}_{32} \E_{32,\tau_{3}} + \tilde{\nu}_{33} \E_{33,\tau_{3}}\right) \mbp_3dp \\ =
&\int  d_1  \left(\tilde{\nu}_{11} +\tilde{\nu}_{12} + \tilde{\nu}_{13} \right) G_1  \mbp_1 dp
+ \int  d_2  \left(\tilde{\nu}_{21} + \tilde{\nu}_{22} + \tilde{\nu}_{23}\right) G_{2}  \mbp_2 dp \\&+ \int  d_3  \left(\tilde{\nu}_{31} + \tilde{\nu}_{32} + \tilde{\nu}_{33} \right) G_{3}  \mbp_3 dp.
\end{split}
\end{align}
These constraints \eqref{eq:moments_A_3} represent first-order optimality conditions associated to the minimization of the  convex function
\begin{align*}
\varphi_{\rm tot}(\alpha_1,\alpha_2,\alpha_3,\alpha_{12},\alpha_{13},\alpha_{23})  = \varphi_1(\alpha_1) + \varphi_2(\alpha_2) + \varphi_3(\alpha_3) + \varphi(\alpha_{12}) + \varphi(\alpha_{13}) + \varphi(\alpha_{23})
\end{align*}
with
\begin{align*}
\varphi_k(\alpha_k) &= \int d_k \tilde{\nu}_{kk} w[\E_{kk,\tau_{k}}] dp + \mu_{kk}\cdot \alpha_k
\end{align*}
and
\begin{align*}
\varphi(\alpha_{kj}) &= \int \left(  d_k \tilde{\nu}_{kj} w[\E_{kj,\tau_{k}}] +  d_j \tilde{\nu}_{jk} w[\E_{jk,\tau_{j}}] \right) dp + \mu_{kj} \cdot \alpha_{kj},
\end{align*}
where
\begin{align*}
\textcolor{black}{ w[\E_{kj,\tau_{k}}] 
=
\begin{cases}
	-\E_{kj,\tau_{k}} \qquad  &\text{for } \tau_{k} = 0, \\
	 \log(1-\E_{kj,+1})\quad &\text{for } \tau_{k} =  +1, \\
	 -\log(1+\E_{kj,-1})\quad &\text{for } \tau_{k} =  -1.
\end{cases} }
\end{align*}
Moreover, $\alpha_k = (\alpha_k^0,\alpha_k^1,\alpha_k^2)^\top$;
\begin{align*}
\mu_{kk} =
\begin{pmatrix}
\mu_{kk}^0 \\ \mu_{kk}^1 \\ \mu_{kk}^2
\end{pmatrix} =
\int d_k \tilde{\nu}_{kk}  G_k \mbp_k dp
\end{align*}
for $k=1,2,3$; for $k\neq j:$ $\alpha_{kj} = (\alpha_{kj}^0,\alpha_{jk}^0,\alpha_{kj}^1,\alpha_{kj}^2)^\top $; and
\begin{align*}
\mu_{kj} =
\begin{pmatrix}
\mu_{kj}^0 \\ \mu_{jk}^0 \\ \mu_{kj}^1 \\ \mu_{kj}^2
\end{pmatrix} &= \int \left[
\begin{pmatrix}
1 \\ 0 \\ p\\ \frac{|p|^2}{2m_k}
\end{pmatrix}
 d_k \nu_{kj}  G_k +
\begin{pmatrix}
0 \\ 1 \\ p \\ \frac{|p|^2}{2m_j}
\end{pmatrix}
 d_j \nu_{jk}  G_j  \right] dp.
\end{align*}

\subsection{Entropy minimization problem of the mixture equilibria}
\label{sec2}
In this section, \textcolor{black}{we prove that the local equilibria in \eqref{parameters} are the unique entropy minimizing solution. We note that we choose the same equilibrium parameter $(b,c)$ in $\mathcal{K}_{12}$ and $\mathcal{K}_{21}$. This is also reflected in the choice of the equilibrium parameters in our numerical scheme.} 
For this, it will be convenient to define the following notations. We denote
\begin{align*}
\begin{split}
 \E_{11} 
    = (e^{-\alpha_{11}\cdot \mathbf{p}_1}+\tau)^{-1}, \quad  \E_{22} 
    = (e^{-\alpha_{22}\cdot \mathbf{p}_2}+\tau')^{-1}, \\ \E_{12} 
    = (e^{-\alpha_{12}\cdot \mathbf{p}_1}+\tau)^{-1}, \quad \E_{12} 
    = (e^{-\alpha_{21}\cdot \mathbf{p}_2}+\tau')^{-1}
    \end{split}
    \end{align*}
where
\begin{align*}
    \mathbf{p}_k(p) := (1,p, \frac{|p|^2}{2m_k})^\top, \quad k=1,2
\end{align*}
and the parameters $(a_{1},a_2, a,b_{1}, b_2, b,c_1, c_2, c_{12}, c_{21})$ can be mapped one-to-one to $\alpha_{kj}=(\textcolor{black}{\alpha^0_{kj}},\alpha^1_{kj}, \alpha^2_{kj}), ~k,j=1,2$. 
As the function \textcolor{black}{ $h_{\tau}$ defined by \eqref{entropy_h} is convex, it follows
\begin{equation}
\label{eq:h_tangent}
h_{\tau}(z) \geq h_{\tau}(y) + \ln(y)(z-y), 
\end{equation}}
for all $y,z>0$ if $\tau=0,-1$ and $0<y,z<1$ if $\tau=+1$.

For interactions between different species, we seek a solution of the entropy minimization problem
\begin{align}
\label{eq:entropy_min_cross_terms}
\min_{g_{1}, g_{2} \in \chi_{12}}  \int  h_{\tau}(g_{1}) dp + \int h_{\tau'}(g_{2}) dp ,
\end{align}
where
\begin{align}
\label{eq:chi12_def}
\begin{split}
\chi_{12}= \Bigg\lbrace (g_{1}, g_{2})~\Big|& ~ g_{1}, g_{2} >0,\,
 (1+ |p|^2) g_{1},\,
 (1+ |p|^2) g_{2}  \in L^1(\mathbb{R}^3),\\
&\int  g_{1} dp = \int  f_1 dp, \quad
\int   g_{2} dp = \int f_2 dp, \\
&\int
\begin{pmatrix}
 p \\  \frac{|p|^2}{2m_1}
\end{pmatrix} (g_{1} - f_1) dp  +
\int \begin{pmatrix}
 p \\ \frac{|p|^2}{2 m_2}
\end{pmatrix} (g_{2} - f_2) dp = 0 \Bigg\rbrace.
\end{split}
\end{align}
Here, $\chi_{12}$ is chosen such that the constraints \eqref{conserv 2} for inter-species collisions are satisfied.
Similar to the case of intra-species collisions, we consider the Lagrange functional $L \colon \chi_{12} \times \mathbb{R} \times \mathbb{R} \times \mathbb{R}^3 \times \mathbb{R} \to \mathbb{R} $
\begin{align*}
\begin{split}
L(g_1, g_2, \alpha_{12}^0,\alpha_{21}^0, \alpha^1, \alpha^2) &= \int  h(g_{1}) dp + \int  h(g_{2}) dp \\
&+\alpha_{12}^0 \int  (g_{1} - f_1 )dp
+ \alpha_{21}^0 \int  (g_{2} - f_2 )dp \\
&+ \alpha^1 \cdot \left( \int   p (g_1 - f_1) dp + \int p (g_2 - f_2) dp \right)\\
&+\alpha^2 \left( \int \frac{  |p|^2}{2m_1} (g_1 - f_1) dp +  \int \frac{  |p|^2 }{2 m_2} (g_2 - f_2) dp \right).\\
\end{split}
\end{align*}%
Any critical point $(\mathcal{K}_{12}, \mathcal{K}_{21}, \lambda_{12}^0,\lambda_{21}^0, \lambda^1, \lambda^2)$ of $L$ satisfies the first-order optimality conditions
\begin{align*}
\frac{\delta L}{\delta g_1}(\mathcal{K}_{12}, \mathcal{K}_{21}, \lambda^0_{12},\lambda^0_{21}, \lambda^1, \lambda^2)
 =   \ln \frac{\mathcal{K}_{12}}{1-\tau \mathcal{K}_{12}} + \lambda_{12} \cdot \mathbf{p}_1(p)  =0, \\
 \frac{\delta L}{\delta g_2}(\mathcal{K}_{12}, \mathcal{K}_{21}, \lambda^0_{12},\lambda^0_{21}, \lambda^1, \lambda^2)
 =   \ln \frac{\mathcal{K}_{21}}{1-\tau \mathcal{K}_{21}} + \lambda_{21} \cdot \mathbf{p}_2(p)   =0,
\end{align*}
where $\lambda_{12} = \textcolor{black}{(\lambda^0_{12}, \lambda^1, \lambda^2)}$ and $\lambda_{21} = \textcolor{black}{(\lambda^0_{21}, \lambda^1, \lambda^2)}$.  Therefore
\begin{align} \label{form1}
\begin{split}
\mathcal{K}_{12} &= (\exp(\lambda_{12} \cdot \mathbf{p}_1(p))+\tau)^{-1}, 
\cr
\mathcal{K}_{21} &= (\exp(\lambda_{21} \cdot \mathbf{p}_2(p))+ \tau')^{-1}. 
\end{split}
\end{align}
We highlight again that we only require conservation of the combined momentum and kinetic energy so that there is only one Lagrange multiplier for the momentum constraint and one Lagrange multiplier for the energy constraint. Hence, $\lambda^1_{12} = \lambda^1_{21}$ and $\lambda^2_{12} = \lambda^2_{21}$ in \eqref{form1}.  When we are in the classical case ($\tau=0$), this restriction is the same as the one used in \cite{HHM}, but more restrictive than the model in \cite{Pirner}.

Theorem 2.1 in \cite{Quantum} shows the existence of functions of the form \eqref{form1} which satisfy the constraints in \eqref{conserv 1} and \eqref{conserv 2}. It follows that these functions are unique minimizer of the corresponding minimization problem.

\begin{theorem}\label{thm2.2}
  $(\mathcal{K}_{12},\mathcal{K}_{21})$ as defined in \eqref{form1} is the unique minimizer of \eqref{eq:entropy_min_cross_terms}.
\end{theorem}

\begin{proof}
	According to \eqref{eq:h_tangent}
	\begin{equation*}
	h_{\tau}(g)
	\geq  h_{\tau_k}(\mathcal{K}_{kj}) + \lambda_{kj} \cdot \mathbf{p}_k (g-\mathcal{K}_{kj}),
	\end{equation*}
	point-wise in $p$, for any measurable function $g$ and $k,j \in \{1,2\}$.
	Therefore it follows that for any measureable functions $g_1$ and $g_2$,
	\begin{align}
	\label{eq:convex_h1h2}
	\int h_{\tau}(g_1) dp + \int & h_{\tau'}(g_2)  dp
		\geq  \int h_{\tau}(\mathcal{K}_{12}) dp
			+ \int  h_{\tau'}(\mathcal{K}_{21})dp  \nonumber \\
			& \quad + \lambda_{12} \cdot\int  \mathbf{p}_1 (g_1-\mathcal{K}_{12})dp
			+ \lambda_{21} \cdot \int   \mathbf{p}_2 (g_2-\mathcal{K}_{21})dp.
	\end{align}
	Since $\lambda^1_{12}=\lambda^1_{21}$ and $\lambda^2_{12}=\lambda^2_{21}$,
	\begin{align*}
\begin{split}
	\lambda_{12} \cdot\int  \mathbf{p}_1& (g_1-\mathcal{K}_{12}) dp
		+ \lambda_{21} \cdot \int  \mathbf{p}_2 (g_2-\mathcal{K}_{21})dp 
	= \lambda^0_{12} \int  (g_1-\E_{12})dp
		+ \lambda^0_{21} \int (g_2-\mathcal{K}_{21})dp \\
		& \quad + \lambda^1_{12} \cdot \left(\int p (g_1-\mathcal{K}_{12}) dp + \int p (g_2-\mathcal{K}_{21}) dp \right)\\& + \lambda^2_{12} \left(\int \frac{|p|^2}{2 m_1} (g_1-\mathcal{K}_{12}) dp
		  + \int  \frac{ |p|^2}{2 m_2} (g_2-\mathcal{K}_{21}) dp \right) .
\end{split}
	\end{align*}
	If $(g_1,g_2)$ and $(\mathcal{K}_{12},\mathcal{K}_{21})$ are elements of  $\chi_{12}$, then the constraints in \eqref{eq:chi12_def} imply that each of the terms above is zero.
	In such cases, \eqref{eq:convex_h1h2} reduces to
	\begin{align*}
	\int h_{\tau}(g_1) dp + \int h_{\tau'}(g_2)  dp
	\geq  \int  h_{\tau}(\mathcal{K}_{12}) dp
	+ \int  h_{\tau'}(\mathcal{K}_{21})dp,
	\end{align*}
	which shows that $(\mathcal{K}_{12},\mathcal{K}_{21})$ solves \eqref{eq:entropy_min_cross_terms}.  Since  $h_{\tau_k}$ is strictly convex, it follows that this solution is unique.

\end{proof}

\vskip2mm
\section*{Acknowledgements}
\textcolor{black}{Gi-Chan Bae is supported by the National Research Foundation of Korea (NRF-2021R1C1C2094843), and}
Marlies Pirner was funded by the Deutsche Forschungsgemeinschaft (DFG, German Research Foundation) under Germany's Excellence Strategy EXC 2044-390685587, Mathematics Münster: Dynamics–Geometry–Structure, by the Alexander von Humboldt foundation and the German Science Foundation DFG (grant no. PI 1501/2-1).

\par{\bf References.}\,

          \end{document}